\documentclass[12pt,reqno]{amsart}
\usepackage{amsmath,amssymb,amsthm,mathtools,calc,verbatim,enumitem,tikz,url,hyperref,mathrsfs,cite,fullpage}
\usepackage{bbm}
\usepackage{stmaryrd}
\usepackage{textcomp}
\usepackage{setspace}
\usepackage{amssymb}
\usepackage{amsthm}
\usepackage{amsmath}
\usepackage{graphicx}
\usepackage{caption}
\usepackage{tikz}
\usetikzlibrary{fit,backgrounds}
\usetikzlibrary{shapes.geometric, arrows.meta}
\usetikzlibrary{intersections}
\usetikzlibrary{calc}
\usepackage{marvosym}
\usepackage{empheq}
\usepackage{latexsym}
\usepackage{fontenc}
\usepackage{color}
\usepackage{hyperref}
\usepackage{cleveref}

\newtheorem{theorem}{Theorem}[section]
\newtheorem{lemma}[theorem]{Lemma}

\newtheorem{observation}[theorem]{Observation}
\newtheorem{proposition}[theorem]{Proposition}

\newtheorem*{claim*}{Claim}

\theoremstyle{definition}
\newtheorem{definition}[theorem]{Definition}
\newtheorem{assumption}[theorem]{Assumption}

\newtheorem*{qu*}{Question}
\theoremstyle{remark}

\newcommand\E{\operatorname{\mathbb{E}}}

\newcommand\cB{\mathcal{B}}

\newcommand\cH{\mathcal{H}}

\newcommand\cM{\mathcal{M}}

\newcommand\cU{\mathcal{U}}

\renewcommand\leq{\leqslant}
\renewcommand\geq{\geqslant}
\renewcommand\le{\leqslant}
\renewcommand\ge{\geqslant}

	\def\Prob{\mathbb{P}}

	\def\G{\Gamma}

	\def\<{\langle }
	\def\>{\rangle }

\begin{document}

\title{Colour-biased Hamilton cycles in randomly perturbed  graphs}
\author{Wenchong Chen\qquad  Xinbu Cheng\qquad  Zhifei Yan}

\address{Nankai University, Weijin Road 94, Nankai District, Tianjin, 300192, P. R. China}\email{2212161@mail.nankai.edu.cn}

\address{IMPA, Estrada Dona Castorina 110, Jardim Bot\^anico,
Rio de Janeiro, 22460-320, Brazil}\email{xinbu.cheng@impa.br, zhifei.yan@impa.br}

\thanks{}

\begin{abstract}

Given a graph $G$ and an $r$-edge-colouring $\chi$ on $E(G)$, a Hamilton cycle $H\subset G$ is said to have $t$ colour-bias if $H$ contains $n/r+t$ edges of the same colour in $\chi$. Freschi, Hyde, Lada and Treglown showed every $r$-coloured graph $G$ on $n$ vertices with $\delta(G)\geq(r+1)n/2r+t$ contains a Hamilton cycle $H$ with $\Omega(t)$ colour-bias, generalizing a result of
Balogh, Csaba, Jing and Pluh\'{a}r. In 2022, Gishboliner, Krivelevich and Michaeli proved that the random graph $G(n,m)$ with $m\geq(1/2+\varepsilon)n\log n$ typically admits an $\Omega(n)$ colour biased Hamilton cycle in any $r$-colouring.

In this paper, we investigate colour-biased Hamilton cycles in randomly perturbed graphs. We show that for every $\alpha>0$, adding $m=O(n)$ random edges to a graph $G_\alpha$ with $\delta(G_\alpha)\geq \alpha n$ typically ensures a Hamilton cycle with $\Omega(n)$ colour bias in any $r$-colouring of $G_\alpha\cup G(n,m)$. Conversely, for certain $G_{\alpha}$, reducing the number of random edges to $m=o(n)$ may eliminate all colour biased Hamilton cycles of $G(n,m)\cup G$ in a certain colouring. In contrast, at the critical endpoint $\alpha=(r+1)/2r$, adding $m$ random edges typically results in a Hamilton cycle with $\Omega(m)$ colour-bias for any $1\ll m\leq n$.
\end{abstract}
	
	\maketitle
\section{Introduction}

The colour bias (or discrepancy) question for spanning structures in edge-colored graphs has attracted growing attention in recent years. This line of work includes results on factors in dense graphs~\cite{BCG}, trees in dense graphs~\cite{HLMP}, Steiner triple systems~\cite{GGS}, and perfect matchings in hypergraphs~\cite{LMX,han2024colour}. In this paper, we specifically focus on colour-biased Hamilton cycles. Let $G$ be a graph on $n$ vertices with an $r$-colouring $\chi$ on $E(G)$ for some $r\geq2$. The graph $G$ is said to contain a Hamilton cycle $H$ with $t$ colour bias, if $H$ has $n/r+t$ edges of the same colour in $\chi$. We study the following quantity:
$$h_r(G)=\min_{\chi\in X_r(G)}\max_{H\in\cH(G)}\{\textup{colour\ bias\ of}\ H\ \textup{in}\ \chi\},$$
where $\cH(G)$ is the collection of Hamilton cycle in $G$, and $X_r(G)$ is the collection of all $r$-colourings on $E(G)$. The study of $ h_r(G) $ was initiated by Balogh, Csaba, Jing and Pluh\'{a}r~\cite{BCJP}, who proved the following theorem in the case $ r = 2 $. This was then extended to general $ r \geq 2 $ by Freschi, Hyde, Lada and Treglown \cite{FHLT}.
\begin{theorem}[Freschi, Hyde, Lada and Treglown]\label{thm:FHLT1}
Let $r\geq2$, $\alpha>(r+1)/2r$ and let $n\in \mathbb{N}$ be sufficiently large. Then for every $G_\alpha$ on $n$ vertices with $\delta(G_\alpha)\geq \alpha n$,
$$h_r(G_\alpha)=\Omega(n).$$
\end{theorem}

Freschi, Hyde, Lada and Treglown actually proved the following more general result: $h_r(G_\alpha)=\Omega(d)$ for every $G_\alpha$ with $\delta(G_\alpha)\geq (r+1)n/2r+d$ (see \cite{FHLT}, Theorem 1.3). The minimal degree condition in \Cref{thm:FHLT1} is tight. Indeed a construction from~\cite{FHLT} shows that there is an $r$-coloured graph $G_\alpha$ with $\delta(G_\alpha)\geq (r+1)n/{2r}$, where every Hamilton cycle uses the same number of edges of each colour. In other words, $h_r(G_\alpha) = 0$. 

In addition to deterministic graphs, we also have a good understanding of how $h_r$ behaves in random graphs. Gishboliner, Krivelevich, and Michaeli \cite{GKM} designed a clever DFS algorithm that finds a long path with almost all edges the same colour in any $r$-colouring of a random graph. This led to the following result.
\begin{theorem}[Gishboliner, Krivelevich and Michaeli]\label{thm:GKM}
For each $r\geq2$,
$$h_r\big(G(n,m)\big)=\Omega(n)$$
holds with high probability if $m\geq (\log n+\log\log n+\omega(1))n/2$.
\end{theorem}

In this paper we will study colour-biased Hamilton cycles in \emph{randomly perturbed graphs}. The randomly perturbed graph is a union of a fixed graph and a random graph on the same vertex set, denoted by $G_\alpha\cup G(n,m)$, which provides a connection between the extremal and random settings. Bohman, Frieze and Martin~\cite{BFM} first introduced this model, and they proved the following theorem on the threshold for Hamilton cycles in $G_\alpha\cup G(n,m)$.

\begin{theorem}[Bohman, Frieze and Martin]\label{thm:BFM}
Let $\alpha>0$. There exists $m=O(n)$ such that for every $G_\alpha$ with $v(G_\alpha)=n$ and $\delta(G_\alpha)\geq\alpha n$, with high probability $G_\alpha\cup G(n,m)$ contains a Hamilton cycle.  
\end{theorem}

In recent years, the study of randomly perturbed graphs has yielded a number of significant results. These include the thresholds of clique factors~\cite{bwt2, HPT}, powers of Hamilton cycles~\cite{benn}, and various spanning structures such as bounded degree trees~\cite{joos2, kks2, bmpp2}. Furthermore, there has been important progress on the Ramsey properties of such graphs, especially for cliques and cycles~\cite{das}. Our first result generalizes \Cref{thm:BFM}, which shows that adding $ O(n) $ random edges suffices to ensure a Hamilton cycle in $ G_\alpha\cup G(n, m) $. We strengthen this by proving that the addition of only $ O(n) $ random edges also guarantees an $ \Omega(n) $-colour-biased Hamilton cycle in $ G_\alpha\cup G(n, m) $ under any $ r $-colouring.

\begin{theorem}\label{thm:0<alpha<3/4}
Let $r\geq2$ and $\alpha>0$. There exists $m=O(n)$ such that for every $G_\alpha$ with $v(G_\alpha)=n$ and $\delta(G_\alpha)\geq\alpha n$, with high probability any $r$-colouring of $E(G_\alpha \cup G(n,m))$ contains a Hamilton cycle with $\Omega(n)$ colour bias. In other words,
$$h_r\big(G_\alpha\cup G(n,m)\big)=\Omega(n)$$
with high probability.
\end{theorem}

Note that by \Cref{thm:FHLT1}, if the graph $ G_\alpha$ has minimum degree $ \alpha n $ with $ \alpha > (r+1)/2r $, then no random edges are required: any $ r $-edge-colouring of $ G_\alpha$ already ensures the existence of a Hamilton cycle with colour bias $ \Omega(n) $. In contrast, for $ 0 < \alpha < (r+1)/2r $, the addition of random edges becomes essential in a strong sense: for any $ m = o(n) $, there exists a construction with minimum degree $ \alpha n $ in which every Hamilton cycle has an equal number of edges of each colour---that is, colour bias zero.

\begin{proposition}\label{prop:upperbound}
Let $r\geq2$ and $0<\alpha<(r+1)/2r$ be constants, and let $m=o(n)$. For all sufficiently large $n$ with $2r\,|\, n$, there exists a graph $G_\alpha$ with $v(G_\alpha)=n$ and $\delta(G_\alpha)\geq\alpha n$ such that 
$$h_r\big(G_\alpha\cup G(n,m)\big)=0$$  holds deterministically. 
\end{proposition}

Interestingly, in the critical case $ \alpha = (r+1)/2r$, the value of $ h_r(G_\alpha \cup G(n, m)) $ exhibits a delicate dependence on the number $m$ of added random edges when $ m = o(n) $. Our main result is the following theorem, which establishes that for every $G_\alpha$ with $\delta(G_\alpha)\geq\alpha n=(r+1)n/2r$ and any $1\ll m\leq n$, the randomly perturbed graph $G_\alpha\cup G(n,m)$ typically contains a Hamilton cycle with $\Omega(m)$ colour bias in any $r$-colouring. Moreover, there also exists a construction such that every Hamilton cycle has $O(m)$ colour bias.

\begin{theorem}\label{thm:alpha=3/4}
Let $r\geq2$ and $\alpha=(r+1)/2r$. For any $1\ll m\leq n$ and every $G_\alpha$ with $v(G_\alpha)=n$ and $\delta(G_\alpha)\geq\alpha n$, with high probability
$$h_r\big(G_\alpha\cup G(n,m)\big)=\Omega(m).$$
Moreover, for any $m\geq 0$ there exists $G_\alpha$ with $\delta(G_\alpha)\geq\alpha n$ such that
$$h_r\big(G_\alpha\cup G(n,m)\big)=O(m)$$
holds deterministically.
\end{theorem}

The proof of \Cref{thm:alpha=3/4} is the most interesting and difficult part of the paper. The key step is a structural characterization of graphs $G_\alpha$ in which every Hamilton cycle has colour bias $ o(m) $; see \Cref{prop:structure}. Specifically, we show that if an $ r $-edge-coloured graph $ G_\alpha $ admits no Hamilton cycle with colour bias $\Omega(m)$, then there exists a subset $W \subseteq V(G_\alpha)$ with $|W|\geq\delta(G_\alpha)$, such that the induced subgraph $G_\alpha[W]$ is almost monochromatic for some colour $c^*$, that is, every matching in $G_\alpha[W]$ without edges of colour $c^*$ has at most $o(m)$ edges. Suppose the condition fails, i.e., there exists a $c$-free matching of size $ \Omega(m) $ within the neighborhood of each vertex $ v \in V(G_\alpha)$ for every colour $c$. In this case, we construct a Hamilton cycle with colour bias $ \Omega(m) $ by forming two cycles whose union covers all of $ V(G_\alpha) $ and which intersect in exactly two vertices; see \Cref{lem:propkey}. One of these cycles, denoted $ F $, has $ \Omega(m) $ edges and is nearly monochromatic; see \Cref{lem:mccycle}. The main task is to construct the second cycle $ H $ so that every vertex of $ F $ can be absorbed into $ H $ in two distinct ways, enabling us to control the colouring of the resulting Hamilton cycle and ensure its colour bias is $ \Omega(m)$. In \Cref{section 4.3}, we show that this structural characterization implies the presence of a Hamilton cycle with colour bias $ \Omega(m) $ in $ G_\alpha \cup G(n, m) $ with high probability, for any $ m = \omega(1) $, regardless of the colouring of the random edges.

The remainder of this paper is organized as follows. In \Cref{section 2}, we give two constructions yielding upper bounds on $h_r(G_\alpha \cup G(n,m))$ for different $\alpha$, showing the number of random edges in \Cref{thm:0<alpha<3/4,thm:alpha=3/4} are tight. In \Cref{section 3}, we apply an absorption method to prove \Cref{thm:0<alpha<3/4}. Finally \Cref{thm:alpha=3/4} is proved in \Cref{section 4}: the structural \Cref{prop:structure} is established in \Cref{section 4.1,section 4.2}, and we finish the proof in \Cref{section 4.3}.

\medskip

\section{Upper bounds: two simple constructions}\label{section 2}

In this section, we present two straightforward constructions which establish upper bounds for $h_r\big(G_\alpha\cup G(m,n)\big)$ for different value of $\alpha$. For convenience we denote $\G_\alpha(n)$ denotes to be the collection of graphs on $n$ vertices with minimum degree $\alpha n$:
$$\G_\alpha(n):=\big\{ G : v(G)=n,\ \delta(G)\geq \alpha n \big\}.$$

First for $0<\alpha<(1+r)/2r$, adding $ m = o(n) $ random edges to $ G_\alpha $ may not yield a Hamilton cycle with $ \Omega(n) $ colour bias in every $ r $-colouring. To be precise, we use the following lemma to show \Cref{prop:upperbound}.

\begin{lemma}\label{lem:upper(0<alpha<3/4)}
Let $ r \geq 2 $ and $ 0 < \alpha < {(1 + r)}/{2r} $, and let $ m = o(n) $. For all sufficiently large $ n $ with $2r\, |\, n$, there exists a graph $ G_\alpha \in \G_\alpha(n) $ such that the following holds. For every graph $F$ with $V(F)=V(G_\alpha)$ and $e(F)=m$, there exists an $ r $-edge-colouring of $ G_\alpha \cup F $ such that every Hamilton cycle has an equal number of edges of each colour.
\end{lemma}

Our construction is based on the following lemma, which is obtained by slightly modifying the Extremal Example 1 in \cite{FHLT}. For completeness we present the proof. We write $\alpha(G)$ to denote the independence number of $G$.

\begin{lemma}\label{lem:uppergeneral}
Let $r\geq 2$ and $n\in\mathbb{N}$, with $2r\, |\, n$. If $G$ is a graph on $n$ vertices with 
$$\alpha(G)\geq \frac{r-1}{2r}\cdot n,$$
then there exists an $r$-colouring $\chi$ of $E(G)$, such that every Hamilton cycle in $G$ contains the same number of edges of each colour in $\chi$.   
\end{lemma}
\begin{proof}
Let $A\subset V(G)$ be an independent set with $|A|=(r-1)n/2r$, and set $B=V(G)\setminus A$. Let's take an equal partition  
$$A=\bigcup_{1\leq i\leq r-1} A_i$$ such that $|A_i|=|A|/(r-1)=n/2r$ for each $1\leq i\leq r-1$. Assign colour $i$ to each edge that is incident to a vertex of $A_i$, and colour $r$ to all edges in $G[B]$. 

If there exists a Hamilton cycle $H\subset G$, then $H$ has exactly $n/r$ edges of each colour. Indeed $H$ has exactly $2|A_i|=2\cdot n/2r=n/r$ edges incident to $A_i$ for each $1\leq i\leq r-1$, hence it contains $n-(r-1)\cdot n/r=n/r$ edges of colour $r$ in $G[B]$, as required. 
\end{proof}

\Cref{lem:upper(0<alpha<3/4)} follows easily from \Cref{lem:uppergeneral}, as shown below.

\begin{proof}[Proof of \Cref{lem:upper(0<alpha<3/4)}]
Let $G$ be the graph with vertex set $V(G)=A\cup B$ and edge set
$$E(G)=\{xy : x\in B,\ y\in V(G)\}.$$
where $A$ and $B$ are disjoint sets with $|A|=(1-\alpha)n$ and $|B|=\alpha n$. Note that $G\in\G_\alpha(n)$, and let $F$ be a graph with $V(F)=V(G)$ and $e(F)=m$. Then $$\alpha(G \cup F) \ge \alpha(G) - e(F) = (1-\alpha) n - m \ge \frac{r-1}{2r} \cdot n$$
where the last inequality holds since $\alpha>(r-1)/2r$, $m=o(n)$ and $n$ is sufficiently large. Therefore by \Cref{lem:uppergeneral}, there exists an $r$-colouring $\chi$ of $E(G_\alpha\cup F)$ such that every Hamilton cycle in $\chi$ has the same number of edges of each colour.
\end{proof}

We now use a similar construction to show that in the critical case $ \alpha=(1+r)/2r $, there exists a graph $ G_\alpha \in \G_\alpha(n) $ and an $ r $-edge-colouring of $ G_\alpha\cup G(n, m) $ such that every Hamilton cycle has colour bias at most $ O(m) $. The following lemma shows that the bound on $m$ in \Cref{thm:alpha=3/4'} is tight up to a constant factor.

\begin{lemma}\label{lem:upperbound (1+r)/2r}
Let $ r \geq 2 $ and $ 0 < \alpha < {(1 + r)}/{2r} $, and let $0\leq m< n/r$. For all sufficiently large $ n $ with $2r\, |\, n$, there exists a graph $ G_\alpha \in \G_\alpha(n) $ such that the following holds. For every graph $F$ with $V(F)=V(G_\alpha)$ and $e(F)=m$, there exists an $ r $-edge-colouring of $ G_\alpha \cup F $ such that every Hamilton cycle has at most $2(r-1)m$ colour bias.
\end{lemma}
\begin{proof}
Let $G$ be the graph with vertex set $V(G)=A\cup B$ and edge set
$$E(G)=\{xy : x\in B,\ y\in V(G)\}.$$
where $A$ and $B$ are disjoint sets with $|A|=(1-\alpha)n$ and $|B|=\alpha n$. Note that $G\in\G_\alpha(n)$, and let $F$ be a graph with $V(F)=V(G)$ and $e(F)=m$. Let's partition $A=\bigcup_{1\leq i\leq r-1} A_i$ such that $|A_i|=n/2r$ for each $1\leq i\leq r-1$. Assign colour $i$ to each edge that is incident to a vertex of $A_i$ for $1\leq i\leq r-1$, and $r$ to the edges of $(G\cup F)[B]$.

Now let $H\subset G\cup F$ be a Hamilton cycle. Note that each edge of $F$ used in $H$ can cause the number of colour $i$ edges to decrease by at most 2, hence $H$ contains at least
$$2|A_i|-2e(F)\geq n/r-2m$$
edges of colour $i$ for each $1\leq i\leq r-1$, and thus at most 
$$n-(r-1)\cdot(n/r-2m)\leq n/r+2(r-1)m$$
edges of the colour $r$. Therefore the colour bias of $H$ is at most $2(r-1)m$ as required.
\end{proof}

Note that \Cref{lem:upperbound (1+r)/2r} proves the latter part of \Cref{thm:alpha=3/4}. We will complete the proof of \Cref{thm:alpha=3/4} by proving the remaining part in \Cref{section 4}.

\medskip

\section{$O(n)$ random edges makes $\Omega(n)$ colour bias}\label{section 3}

In this section, we will prove that adding $O(n)$ random edges to any $G_\alpha\in\G_{\alpha}(n)$ is sufficient to ensure the existence of a Hamilton cycle with $\Omega(n)$ colour bias in any $r$-colouring. To be precise, we will prove the following theorem.

\begin{theorem}\label{thm:0<alpha<3/4'}
For every $\alpha>0$ and $r\geq2$, there exists $C=C(\alpha,r)>0$, such that the following holds. If $G_\alpha\in\G_\alpha(n)$ and $p\geq C/n$, then with high probability every $r$-colouring of $E\big(G_\alpha\cup G(n,p)\big)$ contains a Hamilton cycle with at least $n/2r$ colour-bias. 
\end{theorem}

Our proof employs an absorption method introduced in \cite{RRS} and sprinkling. Let $R_i\sim G(n,p_i)$ (for $1\leq i\leq3$) be three independent random graphs on $V(G_\alpha)$, then $R_1\cup R_2\cup R_3\subset G(n,p)$ where $p= p_1 + p_2 + p_3$, which will be our final random graph. We begin by constructing a path $A$ of length $o(n)$ in $G_\alpha\cup R_1$, such that for any 
$o(n)$ vertices outside $A$, we can use 
$A$ to absorb them - meaning we can extend $A$ into a new path that incorporates all these vertices. Then in $R_2$, we will use some lemmas from \cite{GKM} and \cite{KLS} to find two paths $P_1$ and $P_2$ of length $\Omega(n)$, where $P_1$ is almost monochromatic and the union $P_1\cup P_2$ covers all but $o(n)$ vertices (see \Cref{lem:monochromaticpath,lem:longpath}). We then use $R_3$ to connect the absorber path $A$ and the paths $P_1$ and $P_2$ to form an almost spanning cycle. Finally, we will use $A$ to absorb the remaining vertices to extend the cycle to a Hamilton cycle $H$. The colour bias of $H$ comes from $P_1$, which contains at least $n/r+\Omega(n)$ edges of the same colour, and therefore so does $H$. In the remainder of this section, we will construct two long paths in \Cref{subsec:2path}, prove an absorber lemma in \Cref{subsec:absorber}, and finally gather all the pieces to prove \Cref{thm:0<alpha<3/4'} in \Cref{subsec:0<alpha<3/4'}.

\subsection{Two long paths}\label{subsec:2path}
An important part of our proof is to find two vertex-disjoint paths $P_1,P_2\subset G(n,p)$ of length $\Omega(n)$. Specifically we denote
$$P=(x_1,x_2,...,x_k)$$ 
to be a path $P$ with $V(P)=\{x_1,x_2,...,x_k\}$ and $E(P)=\{x_1x_2,x_2x_3,...,x_{k-1}x_k\}$. First, we will use the following lemma from \cite{GKM} to find an almost monochromatic path $P_1\subset G(n,p)$.

\begin{lemma}[Gishboliner, Krivelevich and Michaeli \cite{GKM}]\label{lem:monochromaticpath}
Let $\varepsilon>0$ and $r\geq 2$. There exists $K=K(\varepsilon,r),C=C(\varepsilon,r)>0$, such that if $p\geq C/n$, the random graph $G(n,p)$ is with high probability such that in any $r$-colouring, there exists a path of length at least $\big(2/(r+1)-\varepsilon\big) n$ in which all but at
most $K$ of the edges are of the same colour.  
\end{lemma}

Next, we will find the second path $P_2$ that is vertex-disjoint with $P_1$, such that the union $P_1\cup P_2$ covers all but $\varepsilon n$ vertices.  

\begin{lemma}\label{lem:longpath}
Let $\beta>\varepsilon>0$, then there exists $C=C(\beta,\varepsilon)>0$ such that the following holds. If $p\geq C/n$, then with high probability every induced subgraph $F\subset G(n,p)$ with $v(F)=\beta n$ contains a path $P\subset F$, which has length at least $v(P)\geq(\beta-\varepsilon)n$.  
\end{lemma}

We need the following definition which is used in the proof of \Cref{lem:longpath}.

\begin{definition}\label{def:mjointed}
A graph $G$ is \emph{$k$-joined} if there is an edge of $G$ between any two   disjoint vertex sets of size $k$.    
\end{definition}

Krivelevich, Lee and Sudakov \cite{KLS} developed a powerful method based on the Depth First Search (DFS) algorithm, to find long paths in general graphs. The following lemma is Lemma 7 in \cite{M}, which is a special case of a result in \cite{KLS}.

\begin{lemma}\label{lem:mjointed}
 If $G$ is $k$-joined, then it has a path with at least $v(G)-2k$ vertices. 
\end{lemma}

Now we can use the $k$-joined property to show \Cref{lem:longpath}.

\begin{proof}[Proof of \Cref{lem:longpath}]
Let $\beta>\varepsilon>0$, $k=\varepsilon n/2$ and $R\sim G(n,p)$ with $p=C/n$ for some sufficiently large constant $C$. Given a set $U\subset V(R)$ of size $|U|=\beta n$, the probability that $R[U]$ is not $k$-joined is at most 
\begin{align}\label{eq:jointed}
\binom{\beta n}{k}^2(1-p)^{k^2}\leq \bigg(\frac{e\beta n}{k}\bigg)^{2k}\cdot e^{-pk^2}\leq \bigg(\frac{2e\beta}{\varepsilon}\bigg)^{\varepsilon n}\cdot e^{-C\varepsilon^2 n/4}\leq 2^{-n},
\end{align}
since $C$ is sufficiently large. By the union bound over all sets $U$ of size $\beta n$, it follows that with high probability the induced subgraph $R[U]$ is $k$-joined for every $U\subset V(R)$ with $|U|=\beta n$, thus by \Cref{lem:mjointed} there exists a path of length at least $(\beta-\varepsilon)n$ in $R[U]$, as required.
\end{proof}

\medskip

\subsection{An absorber lemma}\label{subsec:absorber}

We now formally state the absorber lemma as follows.

\begin{lemma}\label{lem:absorber}
Let $\alpha>0$ and $0<\varepsilon<2^{-14}\alpha^2$. There is 
$C=C(\alpha,\varepsilon)>0$ such that the following holds. If $G_\alpha\in\G_\alpha(n)$ and $R\sim G(n,p)$ with $p=C/n$ and $V(R)=V(G_\alpha)$, then with high probability there exists a path $A\subset G_\alpha\cup R$ with $\varepsilon n/8\leq v(A)\leq4\varepsilon n$ with the following two properties:
\begin{itemize}
    \item [(i)] for every $x\in V(R)$. 
    \begin{align}\label{eq:aaaa1}
e\big(R[N_{G_{\alpha}}(x)]\cap A\big)\geq \varepsilon^2 n;   
\end{align}

\item [(ii)] for every $U\subset V(G_\alpha)\setminus V(A)$ with $|U|\leq \varepsilon^2 n/2$ and every path $A'\subset A$ of length at least $v(A)-2\varepsilon^3 n$, there exists a path $A'_U\subset G_\alpha\cup R$ such that $V(A'_U)=V(A')\cup U$. 
\end{itemize}
\end{lemma}

Our main aim in this subsection is to prove \Cref{lem:absorber}. The idea to construct the absorber is to find a small matching $M \subset R$ first, which nevertheless contains \emph{many} edges of $R[N_{G_{\alpha}}(x)]$ for \emph{every} vertex $x \in V(G_\alpha)$. The absorber $A$ will be a path formed by adding paths of length two or three between consecutive pairs of edges of $M$: we will choose a path of length two in $G$ if the two vertices has $\alpha n/2$ common neighbors in $G$, and otherwise a path of length 3 of the form $e_Ge_Re_G$ (where $e_G\in E(G)$ and $e_R\in E(R)$). We need some quasi-random properties of $R$. Given a graph $G$, let's define $Y(G)$ to be the number of non-isolated edges of $G$, that is
$$Y(G):=\big\{xy\in E(G) : \exists z\in V(G)\ \textup{such\ that}\ yz\in E(G)\big\}.$$

\begin{lemma}\label{lem:properties1}
Let $c>0$ be a constant, and let $R\sim G(n,p)$ with $p=c/n$.
Then
\begin{center}
$cn/4\leq e(R)\leq cn\qquad\textup{and}\qquad Y(R)\leq2p^2n^3\leq 2c^2n$ 
\end{center}
holds with probability at least $1-e^{-\Omega(n^{1/4})}$. Moreover for every $\delta>0$ with $2e/\delta<e^{c\delta/2}$, 
$$R~\ \textup{is}~\ \delta n\textup{-joined}$$
with probability at least $1-e^{-\Omega(n)}$.
\end{lemma}

We also need some quasi-random properties of $G_{\alpha}\cup R$.

\begin{lemma}\label{lem:properties2} 
Let $0<\alpha\leq 1$ and $0<c\leq 2^{-8}$ be constants, let $G_\alpha\in \G_\alpha(n)$ and $R\sim G(n,p)$ with $p=c/n$ and $V(G_\alpha)=V(R)$. Then with high probability $G_\alpha\cup R$ has the following properties:

\begin{itemize}
\item[$(i)$] For every $x\in V(R)$,
$$e\big(R[N_{G_\alpha}(x)]\big)\geq \alpha^2cn/4.$$

\item[$(ii)$] For every $x,y\in V(R)$ with 
$|N_{G_\alpha}(x)\setminus N_{G_\alpha}(y)|\geq \alpha n/2$ and $|N_{G_\alpha}(y)\setminus N_{G_\alpha}(x)|\geq \alpha n/2,$ there exists a matching $M_{x,y}\subset R$ of size
$$e(M_{x,y})\geq 2^{-4}c\alpha^2 n$$
between $N_{G_\alpha}(x)\setminus N_{G_\alpha}(y)$ and $N_{G_\alpha}(y)\setminus N_{G_\alpha}(x)$.
\end{itemize}
\end{lemma}

The proofs of \Cref{lem:properties1,lem:properties2} use some standard probabilistic methods that are technical, hence for completeness we put them in the appendix. Now we are ready to construct the special matching $M\subset R$, that is the cornerstone of our absorber construction.

\begin{lemma}\label{lem:absorbingmatching}
Let $0<\alpha\leq 1$ and $0<\varepsilon \le 2^{-4}\alpha^2$ be constants. If $G_\alpha\in\G_\alpha(n)$ and $R\sim G(n,p)$ with $p=\varepsilon/2n$ and $V(R)=V(G_\alpha)$, then with high probability there exists a matching $M\subset R$ with $\varepsilon n/8\leq v(M)\leq \varepsilon n$, such that 
\begin{align}\label{eq:aaaa}
e\big(R[N_{G_{\alpha}}(x)]\cap M\big)\geq \varepsilon^2 n    
\end{align}
holds for every $x\in V(R)$. 
\end{lemma}

\begin{proof}
Let $c=\varepsilon/2 \le 2^{-4}$, so that $p=c/n$, and observe that, by \Cref{lem:properties1}, with high probability $R$ satisfies
$$cn/4\leq e(R)\leq cn \qquad\textup{and}\qquad Y(R)\leq 2c^2n.$$
By \Cref{lem:properties2}$(i)$, with high probability $G_\alpha \cup R$ satisfies that for each $x\in V(R)$,
$$e\big(R[N_{G_\alpha}(x)]\big)\geq \alpha^2cn/4.$$
For the rest of the proof, we assume these graphs are fixed with these properties. 

Let $M$ be a maximal matching in $R$, and observe that
$$e(M)\geq e(R)-Y(R)\geq cn/4-2c^2n\geq cn/8\geq 2^{-4}\varepsilon n$$
since $Y(R)\leq 2c^2n$ and $c=\varepsilon/2 \le 2^{-4}$, and that
$$ e(M)\leq e(R)\leq \varepsilon n/2.$$
For each $x\in V(R)$ we have 
$$e\big(R[N_{G_\alpha}(x)]\big)\geq\alpha^2\varepsilon n/8\geq 2\varepsilon^2 n$$
since $\alpha^2 \ge 2^4\varepsilon$. Hence the number of edges of $M$ in $R[N_{G_\alpha}(x)]$ is at least
$$e\big(R[N_{G_\alpha}(x)]\big)-Y(R)\geq 2\varepsilon^2 n-2\cdot(\varepsilon/2)^2n\geq \varepsilon^2 n,$$
since $M\subset R$ is a maximal matching in $R$. Thus 
$$e\big(R[N_{G_\alpha}(x)]\cap M\big)\geq \varepsilon^2 n$$
holds for every $x\in V(R)$, as required.
\end{proof}

Finally let's construct the absorber to finish this subsection, as follows.

\begin{proof}[Proof of \Cref{lem:absorber}] Recall that $\alpha>0$ and $\varepsilon<2^{-14}\alpha^2$. Let $G_\alpha\in \G_\alpha(n)$ be a graph on $n$ vertices with $\delta(G_\alpha)\geq \alpha n$. Let $c_1=\varepsilon/2$, $c_2=2^{-8}$ and let $R_i\sim G(n,p_i)$ with $p_i= c_i/n$ (for $i\in \{1,2\}$) be two independent random graphs on $V(G_\alpha)$.  

First by \Cref{lem:absorbingmatching}, with high probability there exists a matching $M\subset R_1$ with $\varepsilon n/8\leq v(M)\leq \varepsilon n$, such that for every $x\in V(G_\alpha)$,
$$e\big(R_1[N_{G_\alpha}(x)]\cap M\big)\geq \varepsilon^2 n.$$

Next since $c_2=2^{-8}$, with high probability $R_2$ and $R_2\cup G_\alpha$ satisfy property (ii) of \Cref{lem:properties2}. To construct our absorber, we simply have to join up the edges of the matching given by \Cref{lem:absorbingmatching} into a path, and show that this path has the claimed absorption property. To form the path, we will use property (ii) of \Cref{lem:properties2}. Let
$$E(M)=\{a_1b_1,a_2b_2,...,a_mb_m\}$$
with $m=e(M)$, that is, $2^{-4}\varepsilon n\leq m\leq \varepsilon n/2$. We claim that we can greedily pick paths of length 2 or 3 to connect all the edges of $M$. Indeed for each pair of $b_i$ and $a_{i+1}$ for $1\leq i\leq m-1$, if $$|N_{G_\alpha}(b_i)\cap N_{G_\alpha}(a_{i+1})|\geq \alpha n/2\geq 4\varepsilon n,$$
then we have at least $4\varepsilon n$ choices of $z_i\in N_{G_\alpha}(b_i)\cap N_{G_\alpha}(a_{i+1})$ to connect $b_i$ and $a_{i+1}$ by the path $(b_i,z_i,a_{i+1})\subset G_\alpha$. Otherwise we know that
$$|N_{G_\alpha}(b_i)\setminus N_{G_\alpha}(a_{i+1})|\geq\alpha n/2\quad\textup{and}\quad |N_{G_\alpha}(a_{i+1})\setminus N_{G_\alpha}(b_i)|\geq\alpha n/2.$$
Hence by (ii) of \Cref{lem:properties2}, there is a matching $M_{b_i,a_{i+1}}\subset R_2$ between $N_{G_\alpha}(b_i)\setminus N_{G_\alpha}(a_{i+1})$ and $N_{G_\alpha}(a_{i+1})\setminus N_{G_\alpha}(b_i)$ of size
\begin{align*}
e(M_{b_i,a_{i+1}})\geq 2^{-4}c_2\alpha^2 n \geq 2^{-12}\alpha^2 n \geq2^{14}\cdot 2^{-10}\varepsilon n \geq 4\varepsilon n, 
\end{align*}
since $\alpha\leq 1$ and $\alpha^2>2^{14}\varepsilon$. Then we have at least $4\varepsilon n$ choices of isolated edges $u_iv_i\subset M_{b_i,a_{i+1}}$ to connect $b_i$ and $a_{i+1}$ by the path $(b_i,u_i,v_i,a_{i+1})\subset R_2\cup G_\alpha$. Now considering both cases, we can greedily connect the entire $E(M)$, since for each step $1\leq i\leq m-1$ we need to forbid at most 
$$2\cdot v(M)\leq 4\varepsilon n$$
choices. Then we can get a path $A$ containing $M$ of length
$$\varepsilon n/8\leq v(A)\leq 4\varepsilon n.$$
Setting $C=c_1+c_2$, we have $R_1\cup R_2\subset R$ where $R\sim G(n,p)$ with $p=C/n$ and $V(R)=V(G_\alpha)$. Hence with high probability we can find the absorber path $A\subset G_\alpha\cup R$ with $\varepsilon n/8\leq v(A)\leq 4\varepsilon n$, such that  \eqref{eq:aaaa1} holds for every $x\in V(G_\alpha)$.

Finally, it remains to check the absorption property of $A$. Let $A'\subset A$ be a path with $v(A')\geq v(A)-2\varepsilon^3 n$, and let $U\subset V(G_\alpha)\setminus V(A)$ with $U=\{x_1,...,x_t\}$ for some $t\leq\varepsilon^2n/2$. Now for every $x_i$ with $1\leq i\leq t$, by \eqref{eq:aaaa1}, there exists at least $$(\varepsilon^2-2\varepsilon^3) n\geq \frac34\cdot\varepsilon^2 n\geq \frac32\cdot |U|$$
vertex-disjoint edges $A'[N_{G_\alpha}(x_i)]$. Hence we can greedily choose one such edge $y_iz_i$ to absorb $x_i$ by adding $y_ix_i$ and $z_ix_i$ to $A'$ and remove $y_iz_i$, for every $1\leq i\leq t$ one by one. Then let $A'_U$ be the final path after we absorb every $x_i\in U$, and we have
$$V(A'_U)=V(A')\cup U$$
as required.
\end{proof}

\subsection{Proof of \Cref{thm:0<alpha<3/4'}}
Let's put \Cref{lem:monochromaticpath,lem:longpath,lem:absorber} together to construct a Hamilton cycle with $\Omega(n)$ colour bias in $G_{\alpha}\cup R$.

\begin{proof}[Proof of \Cref{thm:0<alpha<3/4'}]\label{subsec:0<alpha<3/4'}
Let $0<\varepsilon<2^{-12}\alpha^2$, $C_1$ be the constant as defined in \Cref{lem:absorber}, and $C_2,C_3$ be two sufficiently large constants. Let $G_\alpha\in \G_\alpha(n)$, and let $R_i\sim G(n,p_i)$ with $p_i= C_i/n$ ($1\leq i\leq 3$) be three independent random graphs on $V(G_\alpha)$. 

First by \Cref{lem:absorber}, with high probability the following holds. There exists a path $A\subset G_\alpha\cup R_1$ with $\varepsilon n/8\leq v(A)\leq 4\varepsilon n$, such that 
\begin{align}\label{eq:bbbb}
e\big(R_1[N_{G_\alpha}(x)]\cap A\big)\geq \varepsilon^2 n    
\end{align}
holds for every $x\in V(G_\alpha)$. Moreover for every $U\subset V(G_\alpha)\setminus V(A)$ with $|U|\leq \varepsilon^2 n/2$ and every path $A'\subset A$ of length at least $v(A)-2\varepsilon^3 n$, there exists a path $A'_U\subset G_\alpha\cup R$ such that $V(A'_U)=V(A')\cup U$.

Next recall that $C_2$ is sufficiently large. By \Cref{lem:monochromaticpath}, with high probability for every $r$-colouring on $E(R_2)$, there exists a path $P_1\subset R_2[V(G_\alpha)\setminus V(A)]$ with 
$$v(P_1)=\big(2/(r+1)-\varepsilon^3\big)(n-v(A))\geq \big(2/(r+1)-\varepsilon^3\big)n-v(A),$$
which contains at least 
\begin{align}\label{eq:P1bias}
v(P_1)-K\geq \big(2/(r+1)-\varepsilon^3\big)n-v(A)-K\geq\big(2/(r+1)-5\varepsilon\big)n    
\end{align} 
edges of the same colour, where $K=K(\varepsilon,r)$ is some constant, since $v(A)\leq 4\varepsilon n$ and $K<\varepsilon n$. Indeed \Cref{lem:monochromaticpath} may provide a longer path with length more than $v(P_1)$, then we just need to cut some edges from the end to get $P_1$ of length exactly $\big(2/(r+1)-\varepsilon^3\big)(n-v(A))$. Then let $\beta=1-\big(v(A)+v(P_1)\big)/n$. Since $C_3$ is sufficiently large, by \Cref{lem:longpath}, with high probability there is a path $P_2\subset R_{2}[V(G_\alpha)\setminus V(A\cup P_1)]$ with
$$v(P_2)\geq n-v(A)-v(P_1)-\varepsilon^3 n \geq \big(1-2/(r+1)\big)n.$$ 
Note then the union $P_1\cup P_2$ covers all but $\varepsilon^3 n$ vertices of $V(G_\alpha)\setminus V(A)$. Indeed, setting $T=V(G_\alpha)\setminus V(A\cup P_1\cup P_2)$ we have
$$|T|\leq n-v(P_1)-v(P_2)-v(A)\leq  \varepsilon^3 n.$$

Now we use $R_3$ to connect $A,P_1$ and $P_2$ to be a cycle. Set $\delta=\varepsilon^3$, note that 
$$2e/\delta<e^{C_3\delta/2}$$
since $C_3$ is sufficiently large. Hence by \Cref{lem:properties1} with high probability $R_3$ is $\delta n$-joined. Then we can find an edge in $R_3$ to connect the last $\delta n$ vertices of $A$ with the first $\delta n$ vertices of $P_1$. Similarly, we can connect $P_1$ to $P_2$ and $P_2$ to $A$, since $A,P_1$ and $P_2$ are vertex-disjoint. To be precise, let's write 
$$A=(x_1,...,x_r),\qquad P_1=(y_1,...,y_s)\qquad \textup{and}\qquad P_2=(z_1,...,z_t)$$ 
for $r=v(A), s=v(P_1)$ and $t=v(P_2)$.
Since $R_3$ is $\delta n$-joined, there exists $x_{r'}y_{s'}\in E(G_3)$  connecting $\{x_{r-\delta n},...,x_r\}$ and $\{y_1,...,y_{\delta n}\}$. Then add $x_{r'}y_{s'}$ to $A\cup P_1$ and transfer $\{x_{r'+1},...,x_r\}\subset \{x_{r-\delta n},...,x_r\}$ and $\{y_1,...,y_{s'}\}\subset\{y_1,...,y_{\delta n}\}$ to $T$. Then let's do the same operation for the last $\delta n$ vertices of $P_1$ and the first $\delta n$ vertices of $P_2$ to connect $P_1$ and $P_2$, and then for the last $\delta n$ vertices of $P_2$ and the first $\delta n$ vertices of $A$ to connect $A$ and $P_2$, which makes a cycle $H_0$. Note then $T$ has size
$$|T|\leq \varepsilon^3n +6\delta n\leq 7\varepsilon^3 n<\varepsilon^2 n/2,$$ 
since $\delta=\varepsilon^3$ and $\varepsilon$ is small enough.

We can extend $H_0$ to a Hamilton cycle $H$ on $V(G_\alpha)$ by applying the  absorption property in \Cref{lem:absorber}. Note that the path $A'=A\cap H_0$ has length $v(A')\geq v(A)-2\varepsilon^3 n$, while $T\subset V(G_\alpha)\setminus V(A)$ and $|T|\leq\varepsilon^2 n/2$. Thus, there exists a path $A'_T\subset G_\alpha\cup R_1\cup R_2\cup R_3$, such that $V(A'_T)=V(A')\cup T$. Consequently, $H=H_0\cup A'_T$ forms a Hamilton cycle on $V(G_\alpha)$. Setting $C=C_1+C_2+C_3$ and $R\sim G(n,p)$ with $p=C/n$, we have
$$H\subset G_\alpha\cup R_1\cup R_2\cup R_3\subset G_\alpha\cup R.$$

Finally it remains to check $H$ has at least $n/2r$ colour bias. Indeed by our construction, the path $P'_1:=(y_{\delta n},...,y_{s-\delta n})$
is contained in $H$ which contains at least 
$$\big(2/(r+1)-5\varepsilon\big)n-2\varepsilon^3 n\geq n/r+n/2r$$
edges of the same colours by \eqref{eq:P1bias}, and so does $H$ as required.
\end{proof}

\medskip

\section{Colour bias with changing random edges}\label{section 4}

In this section, we consider the critical endpoint that $\alpha=(r+1)/2r$. In contrast to $0<\alpha<(r+1)/2r$, in this case adding $m$ random edges to every $G\in \G_{\alpha}(n)$ typically yields a Hamilton cycle with at least $\Omega(m)$ colour bias in any $r$-colouring. To be precise, we will prove the following theorem.

\begin{theorem}\label{thm:alpha=3/4'}
Let $r\geq2$, $\alpha=(r+1)/2r$ and $1\ll m< 2^{-5}r^{-1} n$. If $G_\alpha\in\G_\alpha(n)$, then with high probability every $r$-colouring of $E\big(G_\alpha\cup G(n,m)\big)$ has a Hamilton cycle with at least $2^{-20}r^{-1}m$ colour bias.
\end{theorem}

Notice that \Cref{thm:alpha=3/4} can be derived directly from \Cref{lem:upperbound (1+r)/2r} and \Cref{thm:alpha=3/4'}. Our proof of \Cref{thm:alpha=3/4'} relies on a careful analysis of the structure of 
$G_\alpha\in \G_\alpha(n)$ equipped with an $r$-colouring and lacking Hamilton cycles with $\Omega(m)$ colour bias. We find $G_\alpha$ must contain an almost monochromatic induced subgraph of size $\alpha n$ (see \Cref{prop:structure}). This forms the cornerstone of our approach for finding a Hamilton cycle containing more than $n/r+\Omega(m)$ monochromatic edges in $G_\alpha\cup G(n,m)$. In the remainder of this section we will first give the structural \Cref{prop:structure}, and then present the proof of \Cref{thm:alpha=3/4'}.

\medskip

\subsection{A structural proposition }\label{section 4.1}

Our main aim in this and the next subsection is to show the following structural proposition for $G_\alpha$ without Hamilton cycles of $\Omega(m)$ colour bias.

\begin{proposition}\label{prop:structure}
Let $r\geq 2$, let $n\in \mathbb{N}$ be sufficiently large, let $1\leq b\leq 2^{-10} r^{-2} n$, and set $\alpha=(r+1)/2r$. Let $G_\alpha\in\G_{\alpha}(n)$ and let $\chi$ be an $r$-colouring of $E(G_\alpha)$. If each Hamilton cycle in $G_\alpha$ has at most $b$ colour bias, then there exists $U\subset V(G_\alpha)$ with $|U|=\alpha n$ and a colour $c^*\in [r]$ such that
\begin{align}\label{eq:qqqqqq}
e(M)< 2^{9}r\cdot b 
\end{align}
holds for every matching $M\subset G_\alpha[U]$ that has no edges of colour $c^*$.  
\end{proposition}

In the remainder of this and the next subsection, for convenience, let's fix $r\geq 2$, $\alpha=(r+1)/2r$, $n\in \mathbb{N}$ be sufficiently large and $1 \leq b \leq 2^{-10} r^{-2} n$. Let $G_\alpha\in\G_{\alpha}(n)$ and let $\chi$ be an $r$-colouring of $E(G_\alpha)$. Denote 
$$C=[v_1,v_2,...,v_t]$$
to be the cycle $C$ with $V(C)=\{v_1,v_2,...,v_t\}$ and $E(C)=\{v_1v_2,...,v_{t-1}v_t,v_tv_1\}$. To prove \Cref{prop:structure}, we will construct a Hamilton cycle in $G_\alpha$ with $b$ colour bias by assuming \eqref{eq:qqqqqq} is false; to be precise, we will use the following assumption. 

\begin{assumption}\label{ass:111}
For each $c\in [r]$ and $v\in V(G_\alpha)$, there exists a matching $M\subset G_\alpha[N(v)]$ which contains at least $2^9rb$ edges and no edges of colour $c$.
\end{assumption}

Note that the conclusion \eqref{eq:qqqqqq} of \Cref{prop:structure} allows us to assume that this property holds for every set $U \subset V(G_\alpha)$ with $|U| \ge \alpha n$. However, we will only need to use it in the case when $U = N(v)$ for some vertex $v \in V(G_\alpha)$.

Inspired by an idea from \cite{BCJP}, we will use the following lemma to find a Hamilton cycle with $b$ colour bias.

\begin{lemma}\label{lem:propkey}
Let $0\leq s < t \leq n/2$. Assume there exists a colour $c^*\in [r]$ and two cycles $F,H\subset G_\alpha$ where $F=[v_0,...,v_t]$ such that the following properties hold:
\begin{itemize}
\item[$(i)$] All edges in $E(F)$ except $\{v_0v_1,v_0v_t\}$ are of colour $c^*$ in $\chi$.

\item[$(ii)$] The two cycles $F$ and $H$ satisfy 
$$V(F)\cup V(H)=V(G_\alpha),\quad V(F)\cap V(H)=\{v_0,v_t\} \quad\textup{and}\quad E(F)\cap E(H)=\{v_0v_t\}.$$

\item[$(iii)$] There exists distinct edges $x_1y_1,...,x_{t-1}y_{t-1}\in E(H)$ such that $v_ix_i,v_iy_i\in E(G_\alpha)$ for all $1\leq i\leq t-1$. Moreover,
\begin{align}\label{eq:colourbiassum}
\sum_{i=1}^{t-1}\Big( \mathbbm{1}\big[ \chi(v_ix_i)=c^* \big] + \mathbbm{1}\big[ \chi(v_iy_i)=c^* \big] - \mathbbm{1}\big[ \chi(x_iy_i)=c^* \big] \Big)< s. 
\end{align}
\end{itemize}
Then there exists a Hamilton cycle in $G_\alpha$ with colour bias at least $(t-s)/2$.
\end{lemma}
\begin{proof}
We will extend $H$ to two Hamilton cycles $H_1, H_2\subset G_\alpha$, by absorbing the vertices of $F$ in two ways, such that $H_1$ has at least $t-s$ more edges of colour $c^*$ than $H_2$. This implies that one of $H_1$ and $H_2$ must have colour bias at least $(t-s)/2$. 

On the one hand, let $H_1$ be the Hamilton cycle with $E(H_1)=E(F)\cup E(H)\setminus\{v_0v_t\}$. Then $H_1$ has at least $e(F)-2=t-1$ more $c^*$-edges than $H$, since $F$ contains $e(F)-2$ $c^*$-edges. On the other hand, by (iii) we can substitute $x_iy_i$ with a path $(x_i,v_i,y_i)$ for each $1\leq i\leq t-1$, to extend $H$ to a Hamilton cycle $H_2$. Then $H_2$ contains at most $s-1$ more $c^*$-edges than $H$, by \eqref{eq:colourbiassum}. Thus $H_1$ has at least $t-s$
more edges of colour $c$ than $H_2$, as required.
\end{proof}

In the following proof we fix 
$$t= 2^5rb\qquad\textup{and}\qquad s=t/4=8rb.$$
First let's find a colour $c^*\in [r]$ and a near-monochromatic cycle $F$, as follows.
\begin{lemma}\label{lem:mccycle}
There exists a colour $c^*\in[r]$ and a cycle $F\subset G_\alpha$ of length
$$e(F)=t+1,$$
such that all but two consecutive edges of $F$ are of colour $c^*$.
\end{lemma}

\begin{proof}
First by the Erdős–Gallai Theorem on the Turán number of paths, there exists a monochromatic path of length $t$ in some colour $c^*\in [r]$ that contains more than the average number of edges. Then we will close the path into a cycle by selecting a vertex from the common neighborhood of its endpoints, which is larger than the path itself.

To be precise, there exists a colour $c^*\in[r]$ which has at least $n\cdot \alpha n\cdot r^{-1} > {n^2}/{2r}$
edges, since $\alpha>1/2$. By the Erdős-Gallai Theorem, we have
$$\textup{ex}(n,P_t)\leq\frac{(t-2)n}{2}\leq \frac{n^2}{2r},$$
since $t=2^5rb<n/r$. Therefore there exists a monochromatic path $P=(v_1,...,v_t)\subset G_\alpha$ with $v(P)=t$ and of which each edge is of colour $c^*$. Since
$$|N(v_1)\cap N(v_t)|\geq 2\alpha n-n\geq n/r > v(P),$$
there exists $v_0\in N(v_1)\cap N(v_t)$ such that $v_0\notin V(P)$. Set $F$ to be the cycle that 
$$V(F)=V(P)\cup\{v_0\}\quad \textup{and}\quad E(F)=E(P)\cup\{v_0v_1,v_0v_t\},$$ which is the cycle as required.
\end{proof}

Fix the colour $c^*\in [r]$ and the path $F$ in \Cref{lem:mccycle}. We next construct a cycle $H\subset G_\alpha\big[V(G_\alpha)\setminus \big(V(F)\setminus\{v_0, v_t\}\big)\big]$ as required in \Cref{lem:propkey}. We will find a cycle containing some given edges by the following generalization of Dirac's theorem given by P\'{o}sa (see \cite{P}). For completeness we also present the proof.

\begin{lemma}[P\'{o}sa]\label{lem:Posa}
Let $n\geq 3$ and $1\leq \ell\leq n-2$, 
let $G$ be a graph on $n$ vertices, and let $J \subset G$ be a path forest (i.e., a collection of vertex-disjoint paths) with $\ell$ edges. If
$$\delta(G) \geq \frac{n + \ell}{2},$$
then $G$ has a Hamilton cycle $H$ containing $J$ as a subgraph.
\end{lemma}

The key challenge in proving \Cref{prop:structure} is selecting an edge $x_iy_i\in E\big(G_\alpha[N(v_i)]\big)$ for each $v_i\in V(F)$, such that the colour increment inequality \eqref{eq:colourbiassum} holds. Then by \Cref{lem:Posa}, we can find a spanning cycle $H$ in $G_\alpha\big[V(G_\alpha)\setminus \big(V(F)\setminus\{v_0, v_t\}\big)\big]$ that includes all such edges $x_iy_i$. Next we introduce several definitions concerning a general colour $c\in [r]$; these will be instrumental in our arguments, and will in particular be applied to $c^*$.

\begin{definition}
For each $v\in V(G_\alpha)$, we say an edge $xy\in G_\alpha[N(v)]$ is \emph{$c$-good} for $v$ if
\begin{align}\label{eq:f}
 f_c(v,x,y):=\mathbbm{1}\big[ \chi(vx)=c \big] + \mathbbm{1}\big[ \chi(vy)=c \big] - \mathbbm{1}\big[ \chi(xy)=c \big] \leq 0.   
\end{align}    
\end{definition}

For convenience we call $xy\in E(G_\alpha)$ a \emph{$c$-edge} if it is of colour $c$, we call a matching $M\subset G_\alpha$ a \emph{$c$-matching} if all edges in it are of colour $c$, and a \emph{$c$-good matching for $v$} if all edges in $M$ are $c$-good for $v$. We also write
$$N_c(v) = \big\{ u \in V(G_\alpha) : uv \in E(G_\alpha) \,\text{ and }\, \chi(uv) = c \big\}$$
for the \emph{$c$-neighborhood of $v$}, that is, the vertices that are connected to $v$ by a $c$-edge, and also define
$$N_{\Bar{c}}(v):=N(v)\setminus N_{c}(v)\qquad\textup{and}\qquad B(v,c):=G_\alpha\big[ N_{{c}}(v),N_{\Bar{c}}(v) \big]$$
so that $B(v,c)$ is the induced bipartite subgraph of $G_\alpha[N(v)]$ with parts $N_{{c}}(v)$ and $N_{\Bar{c}}(v)$.
\begin{observation}\label{obs:cgoodcondition}
For $v\in V(G_\alpha)$ and $c\in[r]$, an edge $xy$ is $c$-good for $v$ if and only if either 
$$xy\in E\big(G_\alpha[N_{\Bar{c}}(v)]\big)\qquad \textup{or}\qquad xy\in E\big(B(v,c)\big),\, \chi(xy)=c.$$ 
\end{observation}

\begin{proof}
Note that
\begin{itemize}
    \item [(i)] if $xy\in E\big(G_\alpha[N_c(v)]\big)$, then $\chi(vx)=\chi(vy)=c$ and $f_c(v,x,y)\geq1$;

    \item [(ii)] if $xy\in E\big(G_\alpha[N_{\Bar{c}}(v)]\big)$, then $\chi(vx),\chi(vy)\neq c$ and $f_c(v,x,y)\leq0$;  

    \item [(iii)] if $xy\in E\big(B(v,c)\big)$, then $\mathbbm{1}[\chi(vx)=c]+\mathbbm{1}[\chi(vy)=c] = 1$, and therefore either $\chi(xy)=c$ and $f_c(v,x,y)= 0$, or $\chi(xy) \ne c$ and $f_c(v,x,y) = 1$,
\end{itemize}  
see \Cref{fig:cgood}. Since these are the only possibilities, this proves the observation.
\end{proof} 
\begin{figure}[htbp]
\centering
\begin{tikzpicture}
  \tikzset{vertex/.style={circle, fill, inner sep=0pt, minimum size=3pt}}
  \node (v) at (0,0) [vertex, label=above:$v$] {};
  \node[red] at ($(v)+(-0.8,-0.3)$) {$c$}; 
\node[blue] at ($(v)+(1.5,-0.3)$) {$[r]\setminus\{c\}$};
\draw[->, line width=0.6pt] (0,-4.7) -- (0,-4.2);
\draw[->, line width=0.6pt] (3,-1.5) -- (2.5,-2);
\node at (3.3,-1.2) {$c$-good };
\node at ($(0,-5)$) {$c$-good };
  \path[draw, line width=0.6pt, name path=leftEllipse] (-2,-3) ellipse [x radius=1cm, y radius=2cm];
  \path[draw, line width=0.6pt, name path=rightEllipse] (2,-3) ellipse [x radius=1cm, y radius=2cm];
  \path[name path=red1] (v) -- (-2,-2);
  \path[name path=red2] (v) -- (-2,-3);
  \path[name path=red3] (v) -- (-2,-4);
  \path[name path=blue1] (v) -- (2,-2);
  \path[name path=blue2] (v) -- (2,-3);
  \path[name path=blue3] (v) -- (2,-4);
  \path[name intersections={of=red1 and leftEllipse, by=l1}];
  \path[name intersections={of=red2 and leftEllipse, by=l2}];
  \path[name intersections={of=red3 and leftEllipse, by=l3}];
  \path[name intersections={of=blue1 and rightEllipse, by=r1}];
  \path[name intersections={of=blue2 and rightEllipse, by=r2}];
  \path[name intersections={of=blue3 and rightEllipse, by=r3}];
  \draw[red, line width=0.6pt] (v) -- (l1);
  \draw[red, line width=0.6pt] (v) -- (l2);
  \draw[red, line width=0.6pt] (v) -- (l3);
  \draw[blue, line width=0.6pt] (v) -- (r1);
  \draw[blue, line width=0.6pt] (v) -- (r2);
  \draw[blue, line width=0.6pt] (r3) -- (v);
  \draw[blue, line width=0.6pt] (-1.03,-2.4) -- (1.03,-2.4);
\draw[blue, line width=0.6pt] (-1,-2.7) -- (1,-2.7);
  \draw[blue, line width=0.6pt] (-1,-3) -- (1,-3);
  \draw[red, line width=0.6pt] (-1.01,-3.5) -- (1.01,-3.5);
  \draw[red, line width=0.6pt] (-1.06,-3.8) -- (1.06,-3.8);
  \path[name path=h4] (-3,-4.1) -- (3,-4.1);
  \path[name intersections={of=h4 and leftEllipse, by={a4,b4}}];
  \path[name intersections={of=h4 and rightEllipse, by={c4,d4}}];
  \draw[red, line width=0.6pt] (b4) -- (c4);
  \draw[red, line width=0.6pt] (1.55,-1.5) -- (2.55,-2.5);
  \draw[blue, line width=0.6pt] (1.55,-2) -- (2.55,-3);
  \draw[red, line width=0.6pt] (1.55,-2.5) -- (2.55,-3.5);
  \draw[blue, line width=0.6pt] (1.55,-3) -- (2.55,-4);
  \node at (-2,-5.45) {$N_{c}(v)$};
  \node at (2,-5.45) {$N_{\Bar{c}}(v)$};
\end{tikzpicture}
\caption{$c$-good edges for $v$ ~~~}\label{fig:cgood}
\end{figure}
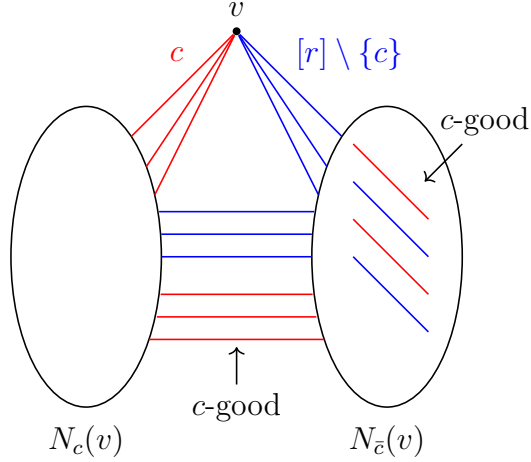

Note that to satisfy condition (iii) of \Cref{lem:propkey} for $c^*$ and $F$, it suffices to find a $c^*$-good edge for all but fewer than $s/2$ vertices in $F$, since $f_{c^*}(v,x,y)\leq 2$ for all $v,x,y\in V(G_\alpha)$. We will greedily select an edge $x_iy_i$ for each $v_i\in V(F)$. To ensure that a $c^*$-good edge for $v_i$ can be chosen disjoint from all $x_jy_j$ with $j\neq i$, as well as from any edge in $F$, it suffices to find a matching $M\subset G_\alpha[N(v)]$ of size
$$e(M)=4t\geq 3\cdot(t+1)\geq 3v(F),$$
in which every edge is $c^*$-good for $v_i$. To this end, we will identify two sets of `bad' vertices ($X_{c}, Y_{c}$ below for general $c\in [r]$, and we will apply them for $c=c^*$) and ensure that every vertex outside these sets admits a $c^*$-good matching of size $4t$.

First for $v\in V(F)$, if nearly all edges in
$G_\alpha[N(v)]$ lie within the $c^*$-neighborhood of $v$, then $v$ may fail to admit a sufficiently large $c^*$-good matching. In general for each $c\in [r]$ and $v\in V(G_\alpha)$, let $G(v,c)$ be the graph with
\begin{align}\label{eq:G(v,c)}
V\big(G(v,c)\big)=N(v)\qquad\textup{and}\qquad E\big(G(v,c)\big)=E\big(G_\alpha[N(v)]\big)\setminus E\big(G_\alpha[N_{c}(v)]\big),
\end{align}
which is the subgraph of $G_\alpha[N(v)]$ after removing all edges in its $c$-neighborhood. Then we have the following definition for $X_{c}$.

\begin{definition}\label{def:Xc}
Let $X_{c}$ be the set of vertices $v\in V(G_\alpha)$ such that every matching $M\subset G(v,{c})$ contains fewer than $8t$ edges, that is
\begin{align}\label{def:S1}
X_{c}:=\big\{v\in V(G_\alpha) : e(M)< 8t,\ \forall M\in  \cM\big(G(v,{c})\big)\big\},   
\end{align}    
where $\cM(G)$ denotes the collection of matchings in $G$.
\end{definition}

Now, for each $v \in V(F)\setminus X_{c^*}$, there is a big enough matching in $G(v,c^*)$. But if most edges in this matching lie in $B(v,c^*)$ and aren’t of colour $c^*$, there may not be enough $c^*$-good edges for $v$. To be precise, let $Y_{c}$ be the following set of vertices for $c\in [r]$.

\begin{definition}
Let $Y_{c}$ be the set of vertices $v\in V(G_\alpha)$ such that $B(v,c)$ contains a matching with at least $4t$ edges without $c$-edges, that is
\begin{align}\label{def:S2}
Y_{c}:=\big\{v\in V(G_\alpha) : \exists M\in \cM\big(B(v,c)\big),\ e(M)\geq 4t\ \textup{and}\ \chi(e)\neq c,\ \forall e\in E(M)\big\}.    
\end{align}
\end{definition}

The point of these definitions is that we can find a sufficiently large $c$-good matching in $G_\alpha[N(v)]$ for each $v\in V(G_\alpha)\setminus(X_{c}\cup Y_{c})$. We record this simple but important fact in the following observation.

\begin{observation}\label{ob:1}
For each $v\in V(G_\alpha)$ and $c\in [r]$, there
exists a matching in $G_\alpha[N(v)]$ with at least $4t$ edges. Moreover, for each $v\in V(G_\alpha)\setminus(X_{c}\cup Y_{c})$, there exists a $c$-good matching in $G_\alpha[N(v)]$ which contains at least $4t$ edges.
\end{observation}
\begin{proof}
First by \Cref{ass:111}, for each $v\in V(G_\alpha)$ there exists a matching $M\subset G[N(v)]$ that has at least $2^9rb>4t$ edges and no edges of colour $c$. Next for each $v\in V(G_\alpha)\setminus(X_{c}\cup Y_{c})$, by \eqref{def:S1} there exists a matching $M\subset G(v,c)$ with
$$e(M)\geq 8t,$$
since $v\notin X_{c}$. And by \eqref{def:S2} every matching $M\subset B(v,c)$ without $c$-edges satisfies
$$e(M)< 4t,$$
since $v\notin Y_{c}$. This means $M$ contains at most $4t$ edges not of colour $c$ in $B(v,c)$, and all the other edges of $M$ are $c$-good for $v$. Hence there is a $c$-good matching $M'\subset M\subset G(v,c)$ such that
\begin{align*}
e(M')\geq e(M)-4t\geq 8t-4t\geq 4t, 
\end{align*}
as required.
\end{proof}

By Observation~\ref{ob:1}, to bound the left hand side of \eqref{eq:colourbiassum} for $F$ and $c^*$, it only remains to bound $|X_{c^*}|$ and $|Y_{c^*}|$. We will use the following lemma to show if either $X_{c^*}$ or $Y_{c^*}$ contains $s/4$ vertices, then $G_\alpha$ contains a Hamilton cycle with $b$ colour bias. 

\begin{lemma}\label{lem:Xc*Yc*}
If there exists $c\in [r]$ such that either 
$$\big|X_{c}\big|\geq s/4\qquad\textup{or}\qquad \big|Y_{c}\big|\geq s/4,$$
then there exists a Hamilton cycle in $G_\alpha$ with colour bias at least $b$.
\end{lemma}

The proof of \Cref{lem:Xc*Yc*} is technical and will be deferred to the next subsection. For now, we assume \Cref{lem:Xc*Yc*} holds and combine it with \Cref{lem:propkey,lem:mccycle,lem:Posa,} and \Cref{ob:1} to prove \Cref{prop:structure}.

\begin{proof}[Proof of \Cref{prop:structure}]
Suppose that the conclusion of the proposition fails to hold, and note that in this case \Cref{ass:111} holds; that is,  for each $c\in [r]$ and $v\in V(G_\alpha)$, there exists a matching $M\subset G_\alpha[N(v)]$ with at least $2^9rb$ edges that has no $c$-edges. We claim that there exists a Hamilton cycle with $b$ colour bias in $G_\alpha$. By \Cref{lem:propkey}, in order to show this it will suffice to find a colour $c^*\in [r]$ and two cycles $F$ and $H$ with 
$$t= 2^5rb\qquad\textup{and}\qquad s=t/4=8rb,$$
since then there exists a Hamilton cycle with $(t-s)/2\geq b$ colour bias in $G_\alpha$.

By \Cref{lem:mccycle}, there exists $c^*\in [r]$ and a cycle $F=[v_0,...,v_t]$ in $G_\alpha$ with length $t+1$, and all edges but $\{v_0v_1,v_0v_t\}$ of $F$ are of colour $c^*$. Let's partition $V(F)=S_1\cup S_2\cup S_3$
\begin{align*}
S_1=X_{c^*}\cap V(F),\quad S_2=\big(Y_{c^*}\setminus X_{c^*}\big)\cap V(F),\quad\textup{and}\quad S_3=V(F)\setminus (S_1\cup S_2).  
\end{align*} 
Let's find an edge in $G_\alpha[N(v_i)]$ for each $v_i\in V(F)$ greedily. On the one hand, by \Cref{ob:1}, for each $v\in S_1\cup S_2$ there exists a matching in $G_\alpha[N(v)]$ with at least $4t$ edges, and for each $v\in S_3$ we can find a matching in $G_\alpha[N(v)]$ in which every edge is $c^*$-good. On the other hand, we just need to avoid at most 
$$3v(F)\leq 3\cdot(t+1)<4t$$
choices for each $v_i$ such that the selected $x_iy_i$ is disjoint from all $x_jy_j$ with $j<i$, as well as from any edge in $F$. Thus by greedily choosing one edge in the corresponding matching for each $v_i\in V(F)$, we get a collection of vertex-disjoint edges 
$$\{x_0y_0,x_1y_1,...,x_ty_t\},$$
in which $x_iy_i$ is $c^*$-good for $v_i$ if $v_i\in S_3$. 

Let's find a spanning cycle $H\subset G':=G_\alpha\big[V(G_\alpha)\setminus \big(V(F)\setminus\{v_0,v_t\}\big)\big]$. Denote 
$$T=\{v_0v_t\}\cup \{x_1y_1,...,x_{t-1}y_{t-1}\},$$
and then $|T|=t$. By \Cref{lem:Posa}, there exists a spanning cycle $H\subset G'$ such that $T\subset E(H)$, since $t \le n/4r$, and hence 
$$\delta(G') \ge \alpha n-v(F)\geq n/2+n/2r-t\geq n/2 + |T|.$$

It remains to check \eqref{eq:colourbiassum} in (iii). Observe that the function $f_{c^*}(v_i,x_i,y_i)$ in \eqref{eq:f} is at most 2 for $v_i\in S_1\cup S_2$, and non-positive for $v_i\in S_3$ since the edge $x_iy_i$ is $c^*$-good for $v_i$. Then either the sum in \eqref{eq:colourbiassum} satisfies
$$\sum_{1\leq i\leq t-1}f_{c^*}(v_i,x_i,y_i)\leq 2\big(|S_1| + |S_2|\big)< s,$$
as required, or 
$$|X_{c^*}| + |Y_{c^{*}}| \geq |S_1| + |S_2| \ge s/2,$$
in which case \Cref{lem:Xc*Yc*} implies that there exists a Hamilton cycle in $G_\alpha$ with colour bias at least $b$, as required. 
\end{proof}

\medskip

\subsection{Proof of \Cref{lem:Xc*Yc*}}\label{section 4.2}
In this subsection, we will prove \Cref{lem:Xc*Yc*} to complete the proof of \Cref{prop:structure}. Recall that $G_\alpha\in \G_\alpha(n)$ for $\alpha=(r+1)/2r$, and we choose 
$$t=2^5r\cdot b\qquad \textup{and} \qquad s=t/4$$
for some $1\leq b\leq 2^{-10} r^{-2} n$. By \Cref{ass:111}, for each $v\in V(G_\alpha)$ and each colour $c\in [r]$, there exists a matching $M\subset G_\alpha[N(v)]$ with at least 
$$2^9r\cdot b=16t$$
edges that has no edges of colour $c$. If $X_c$ (or $Y_c$) contains at least $s/4$ vertices, we may construct a collection of vertex-disjoint \emph{bowties} using vertices from $X_c$ (or $Y_c$). These bowties help us to make a Hamilton cycle with $b$ colour bias. We now define bowties precisely.

\begin{definition}
A subgraph $B\subset G$ is called a \emph{bowtie} if it consists of two triangles that intersect only at a single vertex $z$, referred to as the \emph{center}. An edge $e\in E(B)$ is a \emph{center-edge} if it has 
$z$ as an endpoint; otherwise, it is called a \emph{side-edge}.    
\end{definition}

Let $\cB$ be a collection of vertex-disjoint bowties in $G_\alpha$, and let $B_0 \subset V(G_\alpha)$ denote the set of centers of the bowties in $\cB$. By \Cref{lem:Posa}, there exists a Hamilton cycle $H_0\subset G_\alpha[V(G_\alpha)\setminus B_0]$, which contains both side-edges of each bowtie in $\cB$ (see \Cref{fig:bowtie}). To extend $H_0$ to a Hamilton cycle of $ G_\alpha $, we absorb each vertex $ v \in B_0 $ by replacing one of the two side-edges of the corresponding bowtie with the two center-edges connecting $v$ to $H_0$. The choice of side-edges thus influences the colour bias of the resulting Hamilton cycle, depending on the colours of the side-edges and the center-edges used in the substitution. We next consider two different cases, corresponding to $|X_c|\geq s/4$ and $|Y_c|\geq s/4$ respectively.

\begin{figure}[htbp]
\centering
\begin{tikzpicture}[scale=1.2, every node/.style={circle, draw, fill=white, inner sep=1pt}]
\begin{scope}[rotate=90]

\def\n{4}

\foreach \i in {0,...,3} {
    \pgfmathsetmacro{\y}{-\i*2}
    
    \node (c\i) at (0,\y) {};
    \node (a\i) at (-1,\y+0.6) {};
    \node (b\i) at (-1,\y-0.6) {};
    \node (d\i) at (1,\y+0.6) {};
    \node (e\i) at (1,\y-0.6) {};
    \draw[very thick]  (a\i) -- (b\i) -- (c\i) -- (a\i);
    \draw[very thick]  (d\i) -- (e\i) -- (c\i) -- (d\i);
}
\draw[rounded corners=35pt] (-1,2) rectangle (1,-8);
\end{scope}
\end{tikzpicture}
\caption{A cycle containing the side-edges of bowties ~~~}\label{fig:bowtie}
\end{figure}
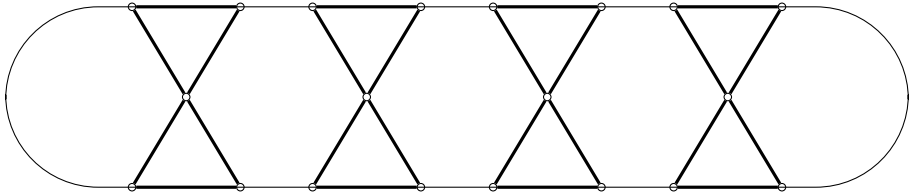

First, if there exists a colour $c \in [r]$ such that $|X_{c}|\geq s/4$, we can find a collection of vertex-disjoint bowties of \emph{$c$-type I} containing the vertices in $|X_c|$, defined as follows.
 
\begin{definition}
A bowtie $B$ is of \emph{$c$-type I} for $c\in [r]$, if it has 4 center-edges of colour $c$, and 2
side-edges whose colours are distinct from each other.
\end{definition}

Note that for each $ v \in X_c $ and each colour $ c' \in [r] $, by \Cref{ass:111}, the graph $ G_\alpha[N(v)] $ contains a matching of size at least $ 16t $ that avoids $ c' $-coloured edges. Moreover, such a matching includes at least $ 8t $ edges within the $ c $-neighborhood of $ v $, since any matching in $ G(v, c) $ (i.e., $ G_\alpha[N(v)] $ outside the $ c $-neighborhood of $ v $; see \eqref{eq:G(v,c)}) contains at most $ 8t $ edges by Definition~\ref{def:Xc}. Hence, we have sufficient flexibility to greedily construct a collection of vertex-disjoint bowties of $ c $-Type~I based on $ X_c $. To be precise, we have the following lemma.

\begin{lemma}\label{lem:bowtie01}
Let $c\in [r]$. If
$$|X_c|\geq s/4,$$
then there exists a collection of $s/4$ vertex-disjoint bowties of $c$-type I. 
\end{lemma}
\begin{proof}
We may assume $X_c=\{x_1,...,x_{s/4}\}$, and then we will find a bowtie of $c$-type I with center $x_i$ for all $x_i\in X_c$ greedily. For each $x\in X$, there exists a matching $M_1\subset G_\alpha[N_{c}(x)]$ with at least 
\begin{align}\label{eq:qwer}
e(M_1)\geq 16t-8t\geq 8t   
\end{align}
edges which doesn't have $c$-edges. Indeed, by \Cref{ass:111}, there is a matching of size $ 16t $ in $ G $ avoiding colour $ c $, and by \eqref{def:S1}, at most $ 8t $ of its edges lie outside $ G_\alpha[N_c(x)] $. Hence, at least $ 8t $ such edges remain within $ G_\alpha[N_c(x)] $, forming $ M_1 $. Choose an arbitrary edge $ e_1 \in M_1 $, and let $ c_1 $ be its colour. Applying the same reasoning, there exists another matching $ M_2 \subset G_\alpha[N_c(x)] $ consisting of edges that avoid colour $ c_1 $, with
$$e(M_2) \geq 8t,$$
by \Cref{ass:111} and \eqref{def:S1} again. Selecting an edge $ e_2 \in M_2 $ which is disjoint with $e_1$, we obtain a $ c $-type I bowtie centered at $ x $, with side-edges $ e_1 $ and $ e_2 $. Thus, for each $ x \in X_c $, there are at least 
$$ \min\{e(M_1),\ e(M_2)-2\}\geq 8t-2$$
such bowties can be chosen.

We now construct, greedily, a family of vertex-disjoint $ c $-type I bowties centered at vertices $ x_i \in X_c $, for $ 1 \leq i \leq s/4 $. To ensure the bowtie centered at $ x_i $ is disjoint from all previously chosen bowties (with centers $ x_j $ for $ j < i $), it suffices to avoid at most
$$5 \cdot |X_c| \leq 5 \cdot s/4 \leq 2s $$
choices. Since $ 2s \leq 8t-2 $, the greedy selection ensures a vertex-disjoint bowtie can be found for every $ x_i \in X_c $.
\end{proof}

We claim if there exists $s/4$ vertex-disjoint bowties of $c$-type I in $G_\alpha$, then it contains a Hamilton cycle with $b$ colour bias.

\begin{lemma}\label{lem:bowtie1}
Let $0\leq\ell\leq 2^{-5}n$. If $G_\alpha$ contains at least $\ell$ vertex-disjoint bowties of \emph{$c$-type I}, then it has a Hamilton cycle with colour bias at least $\ell/2r$.  
\end{lemma}
\begin{proof}
By the pigeonhole principle, there exists a colour $c'\in[r]$ ($c'$ could be $c$), such that the corresponding subset of bowties, denoted $\cB_{c'}$, contains at least
$\ell/r$ vertex-disjoint bowties, each featuring a side-edge of colour $c'$. We may assume $|\cB_{c'}|=\ell/r$. Let $Z_{c'}\subset V(G_\alpha)$ be the set of vertices contained in some bowtie in $\cB_{c'}$, then $$|Z_{c'}|=5\cdot |\cB_{c'}|=5\ell/r.$$

Next by \Cref{lem:Posa}, there exists a spanning cycle $H\subset G_\alpha[V(G_\alpha)\setminus Z_{c'}]$ which contains the side edges of every bowtie in $\cB_{c'}$, since the minimum degree of $G_\alpha[V(G_\alpha)\setminus Z_{c'}]$ is at least 
$$\delta(G_\alpha)-|Z_{c'}|\geq \frac{n}{2}+\frac{n}{2r}-\frac{5\ell}{r}\geq \frac{n}{2} + \frac{5\ell}{r} =\frac{n}{2}+|Z_{c'}|,$$
where the second inequality holds since $n/2r>10\ell/r$, by $\ell\leq2^{-5}n$.

We can extend $H$ by adding the centers of bowties to $H$ in 2 ways: Either remove the side edge of $c'$ from the bowtie and incorporate the two center edges adjacent to it into $E(H)$, or perform the same operation for the opposite side edge. Let $H_1$ be the Hamilton cycle obtained by extending 
$H$ while removing the side edge of 
$c'$ from each bowtie in $\cB_{c'}$. Similarly, let $H_2$ be the Hamilton cycle extending $H$ while removing the side-edge opposite to the 
$c'$-side in each bowtie of $\cB_{c'}$. Note that $H_1$ has $\ell/r$ more edges of $c'$ than $H_2$, which implies a Hamilton cycle with colour bias at least $\ell/2r$.
\end{proof}

Note then \Cref{lem:Xc*Yc*} follows by \Cref{lem:bowtie01,lem:bowtie1} easily, if we suppose $|X_c|\geq s/4$. Next let's consider the case where $|Y_c|\geq s/4$. We will identify another collection of vertex-disjoint bowties containing vertices 
$v\in Y_c$, which are different from $c$-type I.

\begin{definition}
A bowtie $B$ is of \emph{$c$-type II} for $c\in [r]$, if it has two
side-edges of distinct colours, one of which is a $c$-edge $e$, and both center edges incident to $e$ are not of colour $c$.
\end{definition}

We next find a collection of vertex-disjoint bowties of $c$-type II based on $Y_c$.

\begin{lemma}\label{lem:bowtie02}
If there exists $c\in [r]$ such that 
$$|Y_c|\geq s/4,$$
then there exists at least $s/4$ vertex-disjoint bowties of $c$-type II.
\end{lemma}
\begin{proof}
We may assume $Y_{c}=\{y_1,...,y_{s/4}\}$, and then we will find a bowtie of $c$-Type II containing $y_i$ for all $y_i\in Y_c$ greedily. On the one hand, for each $y\in Y_c$, by \eqref{def:S2} there exists a matching $M_1\subset B(y,c)$ without $c$-edges satisfies
$$e(M_1)\geq 4t.$$
Let's pick an edge $zw\in E(M_1)$
such that $yz$ is a $c$-edge while $yw$ is not. Fix $w$, and let's find a bowtie with center $w$. By \Cref{ass:111}, there exists a matching $M_2\subset G_\alpha[N(w)]$ without $c$-edges such that 
$$e(M_2)\geq 16t.$$
Arbitrarily pick an edge $ab\in M_1$ which is disjoint with $yz$, we get a bowtie on $\{y,z,w,a,b\}$ with center $w$, that is of $c$-Type II since $yz$ is a $c$-edge (as a side edge) while $yw,zy$ (as center edges connecting $yz$) and $ab$ (another side edge) are not. Note that we have at least
$$\min\{e(M_1),\ e(M_2)-2\}\geq 4t$$
choices to find such a bowtie. On the other hand, for each $1\leq i\leq s/4$, to make sure the bowtie containing $y_i$ is vertex-disjoint with previous bowties containing $y_j$ that $j<i$, we need to avoid at most
$$5\cdot |Y_c|\leq 5\cdot s/4\leq 2s.$$
Since $2s<4t$, we can greedily find a vertex-disjoint bowtie of $c$-type II for each $y_i\in Y_c$. 
\end{proof}

We claim if there exists $s/4$ vertex-disjoint bowties of $c$-type II in $G_\alpha$, then it has a Hamilton cycle with $b$ colour bias as well. The following lemma is similar to \Cref{lem:bowtie1}.

\begin{lemma}\label{lem:bowtie2}
Let $0\leq\ell\leq 2^{-5}n$ and $c\in [r]$. If $G_\alpha$ contains at least $\ell$ vertex-disjoint bowties of \emph{$c$-type II}, then it has a Hamilton cycle with colour bias at least $\ell/2$. 
\end{lemma}

\begin{proof}
Let $\cB$ be the collection of these vertex-disjoint bowties of $c$-type II, and let $Z\subset V(G_\alpha)$ be the set of vertices of bowties in $\cB$. By \Cref{lem:Posa}, there exists a spanning cycle $H\subset G_\alpha[V(G_\alpha)\setminus Z]$ which contains the side edges of every bowtie in $\cB$. We can also extend $H$ to two Hamilton cycles, by substituting the side-edge of colour $c$ or the opposite side-edge to a path of length 2 containing the center for each bowtie in $\cB$. Then one of these two Hamilton cycles contains at least $t$ more edges of colour $c$ than the other, which implies a Hamilton cycle with colour bias at least $\ell/2$, as required.
\end{proof}

Now put \Cref{lem:bowtie01,lem:bowtie1,lem:bowtie02,lem:bowtie2} together, we are ready to prove \Cref{lem:Xc*Yc*}.

\begin{proof}[Proof of \Cref{lem:Xc*Yc*}]
First if $|X_c|\geq s/4$, there exists $s/4$ disjoint bowties of $c$-type I by \Cref{lem:bowtie01}. This implies a Hamilton cycle in $G_\alpha$ with colour bias $s/8r\geq b$, by choosing $\ell = s/4$ and applying \Cref{lem:bowtie1}. Similarly if $|Y_c|\geq s/4$, there exists $s/4$ disjoint bowties of $c$-type II by \Cref{lem:bowtie02}. And by choosing $\ell=s/4$ and applying \Cref{lem:bowtie2}, there exists a Hamilton cycle in $G_\alpha$ with colour bias $s/8\geq b$, finishing the proof.
\end{proof}

\medskip

\subsection{Proof of \Cref{thm:alpha=3/4'}}\label{section 4.3}

In this subsection we will apply the structural \Cref{prop:structure} to show \Cref{thm:alpha=3/4'}. We will use a quasi-random property of $G\cup G(n,m)$ as follows.

\begin{lemma}\label{lem:3/4quasiproperty}
Let $0<\alpha\leq1$, $\beta\geq 1-\alpha$ be constants and $1\ll m\leq 2^{-5}n$. If $G_\alpha\in\G_{\alpha}(n)$ and $R\sim G(n,p)$ with $p=m/n^2$ and $V(R)=V(G_\alpha)$, then with high probability there exists a matching $M_U\subset (G_\alpha\cup R)[U]$ such that
$$e(M_U)\geq 2^{-5}\beta ^2\cdot m$$
for every $U\subset V(G_\alpha)$ with $|U|=\beta n$.
\end{lemma}

To show \Cref{lem:3/4quasiproperty} we need the following simple lemma.

\begin{lemma}\label{lem:Rjanson}
Let $0<\beta\leq1$ be a constant and $1\ll m\leq 2^{-5}n$, and let $R\sim G(n,p)$ with $p=m/n^2$. Fix an arbitrary $U\subset V(R)$ with $|U|=\beta n$, then with high probability there exists a matching $M_U\subset R[U]$ of size
$$e(M_U)\geq 2^{-4}\beta^2\cdot m.$$
\end{lemma}

The proof of \Cref{lem:Rjanson} uses simple probabilistic techniques, and for completeness we include the proof in the Appendix. To prove \Cref{lem:3/4quasiproperty}, We begin by selecting a maximal collection of $\beta n$-subsets $U_i \subset V(G_\alpha)$ such that each $U_i$ contains no $\Omega(m)$-matching in $G_\alpha$, and for every pair $U_i, U_j$, their respective maximal independent sets are disjoint. Note that these sets are nearly disjoint, so the collection has constant size, roughly $1/\beta$. We then add random edges to $G_\alpha$ and apply \Cref{lem:Rjanson} to guarantee an $\Omega(m)$-matching within each $U_i$. Finally, we observe that any $\beta n$-subset almost completely overlap with some $U_i$, and hence must contain an $\Omega(m)$-matching. We now make this precise and prove \Cref{lem:3/4quasiproperty}.

\begin{proof}[Proof of \Cref{lem:3/4quasiproperty}]
Given $G_\alpha\in\G_\alpha(n)$, let $\cU$ be the collection of subsets $U\subset V(G_\alpha)$ with $|U|=\beta n$ such that every matching in $G_\alpha[U]$ has fewer than $2^{-5}\beta^2m$ edges. We may assume $\cU\neq \emptyset$. Let $I(U)$ be the largest independent set in $G_\alpha[U]$ for $U\in\cU$, then 
\begin{align}\label{eq:I(U)}
|I(U)|\geq |U|-2\cdot 2^{-5}\beta^2m\geq \beta n-2^{-4}\beta^2m.    
\end{align}
Let $\cU_0=\{U_1,U_2,..,U_\ell\}$
be a maximal collection of $U\in\cU$ such that $$I(U_i)\cap I(U_j)=\emptyset$$
for each pair $U_i,U_j\in\cU_0$ with $i\neq j$. Note by \eqref{eq:I(U)}, the number of sets in $\cU_0$ is
$$|\cU_0|\leq\left\lceil \frac{n}{|I(U_i)|}\right\rceil \leq\left\lceil \frac{n}{\beta n-2^{-4}\beta^2m}\right\rceil\leq 2\lceil\beta^{-1}\rceil.$$

Next let's consider $G_\alpha\cup R$. For each $U_i\in\cU_0$, by \Cref{lem:Rjanson} there exists a matching $M_{U_i}\subset (R\cup G_\alpha)[U_i]$ 
of size
\begin{align}\label{eq:e(MUi)}
e(M_{U_i})\geq 2^{-4}\beta^2m    
\end{align}
with high probability. By taking a union bound over all $U_i\in \cU_0$, we know that $(G_\alpha\cup R)[U_i]$ contains a matching with at least $2^{-4}\beta^2m$ edges for all $U_i\in \cU_0$ with high probability. In the remainder of this proof, we consider $R\cup G_\alpha$ satisfying this property.

Now arbitrarily pick a set $U\in\cU\setminus \cU_0$. We claim there exists $U_i\in \cU_0$ such that
\begin{align}\label{eq:I(U)cap I(Ui)|geqbeta n-2^{-7}beta3m}
|I(U)\cap I(U_i)|\geq\beta n-2^{-6}\beta^2m.    
\end{align}
Indeed, there exists $U_i\in\cU_0$ such that $I(U)\cap I(U_i)\neq\emptyset$, since $\cU_0$ is a maximal collection that $I(U_j)\cap I(U_k)=\emptyset$. If \eqref{eq:I(U)cap I(Ui)|geqbeta n-2^{-7}beta3m} doesn't hold, then the neighborhood of each $x\in I(U)\cap I(U_i)$ doesn't intersect to $I(U_i)\cup I(U)$ in $G_\alpha$. Hence by \eqref{eq:I(U)} and \eqref{eq:I(U)cap I(Ui)|geqbeta n-2^{-7}beta3m},
\begin{align*}
d_{G_\alpha}(x)\leq n-|I(U_i)\cup I(U)|&\leq n-\big(|I(U_i)|+|I(U)|-|I(U)\cap I(U_i)|\big)\\
&\leq (1-\beta)n - 2^{-6}\beta^2m,
\end{align*} 
which makes a contradiction since $d_{G_\alpha}(x)\geq \alpha n\geq(1-\beta)n$. Now by \eqref{eq:e(MUi)} and \eqref{eq:I(U)cap I(Ui)|geqbeta n-2^{-7}beta3m}, 
\begin{align*}
|V(M_{U_i})\cap U|\geq v(M_{U_i})-|U_i\setminus U|\geq v(M_{U_i})-\big(|U_i|-|I(U)\cap I(U_i)|\big)\geq \frac{3}{4}\cdot v(M_{U_i}).
\end{align*}
Therefore $(R\cup G_\alpha)[U\cap U_i]$ contains at least $$\frac12\cdot e(M_{U_i})\geq 2^{-5}\beta^2m$$ edges of $M_{U_i}$, that forms a matching, as required.
\end{proof}

Finally put the structural \Cref{prop:structure} and \Cref{lem:Posa,lem:3/4quasiproperty} together, we are ready to finish the proof of \Cref{thm:alpha=3/4'}.

\begin{proof}[Proof of \Cref{thm:alpha=3/4'}]
Let $\alpha=(r+1)/2r$ and $\beta=1-\alpha$, let $G_\alpha\in \Gamma_\alpha(n)$ be a graph with $v(G_\alpha)=n$ and $\delta(G_\alpha)\geq \alpha n$, and let $G(n,m)$ be the random graph on $V(G_\alpha)$ with $m$ random edges. By \Cref{lem:3/4quasiproperty}, with high probability the union $G_\alpha\cup G(n,m)$ has the following property: for every subset $W\subset V(G_\alpha)$ with $|W|=\beta n$, there exists a matching $M\subset \big(G_\alpha\cup G(n,m)\big)[W]$ with
\begin{align}\label{eq:e(M0)}
e\big(M\big)> 2^{-5}\beta^2m\geq 2^{-5}\cdot 2^{-2}\cdot(1-r^{-1})^2\geq 2^{-9}m.  
\end{align} 
Indeed we convert $G(n,p)$ to $G(n,m)$ (see Proposition 1.13 in \cite{JLR}) since the above property is monotone. Next let's consider a fixed $F= G_\alpha\cup G(n,m)$ which satisfies \eqref{eq:e(M0)}.

Let $\chi$ be an arbitrary $r$-colouring of $E(F)$ and let
$$b=2^{-20}r^{-1}m.$$
If the restriction of $\chi$ to $G_\alpha$ doesn't have a Hamilton cycle with $b$ colour bias, then by \Cref{prop:structure} there exists a colour $c^*\in [r]$ and a subset $U\subset V(G_\alpha)$ with $|U|=\alpha n$, such that every maximal matching $M'\subset G_\alpha[U]$ without $c^*$-edges satisfies
\begin{align}\label{eq:e(M1)}
e(M')\leq 2^9rb\leq 2^{-11}m. 
\end{align} 
Set $W=V(G_\alpha)\setminus U$, then $|W|=\beta n$. Hence by \eqref{eq:e(M0)}, there exists a matching $M\subset F[W]$ containing at least $2^{-9}m$ edges. Then there exists a Hamilton cycle $H\subset F$ which contains $M$ as a subgraph by \Cref{lem:Posa}, since $v(M)\leq 2e(G(n,m))\leq 2m\leq n/2r$ and
$$\delta(F)\geq \frac{n}{2}+\frac{n}{2r}\geq \frac{n}{2}+v(M).$$ 

Now we claim the Hamilton cycle $H$ has at least $b$ colour bias. Observe that the edges of $H$ not coloured by $c^*$ can be partitioned into three parts: the edges in $H \cap F[U]$, those in $H \cap F[W]$, and those in the bipartite subgraph $H \cap F[W, U]$.

In $H\cap F[U]$, there are at most 
\begin{align}\label{eq:tttt1}
q\leq 2e(M')    
\end{align}
edges that are not coloured by $c^*$, otherwise there would be a matching in $F[U]$ with more than $e(M')$ edges without $c^*$-edges.

The subgraph $H\cap F[W]$ is a linear path forest which contains
\begin{align}\label{eq:tttt2}
d\geq e(M)    
\end{align}
edges, since $M\subset (H\cap F[W])$. Note it's possible that all edges of $H\cap F[W]$ are not of $c^*$.

The bipartite subgraph $H\cap F[W,U]$ contains at most
\begin{align}\label{eq:tttt3}
2d + 2\cdot\big(|W|-2d\big)    
\end{align}
edges. Indeed, the linear forest $H\cap F[W]$ contains at most $d$ disjoint paths and up to $|W| - 2d$ isolated vertices. Each isolated vertex is incident to two edges in $H\cap F[W,U]$, while each path endpoint (of which there are at most $2d$) is incident to one such edge. All remaining vertices are not incident to any edge in $H\cap F[W,U]$. Therefore, in total, $H\cap F[W,U]$ contains at most $2(|W| - 2d) + 2d$ edges. It's also possible that all these edges don't use $c^*$.

Putting \eqref{eq:tttt1}, \eqref{eq:tttt2} and \eqref{eq:tttt3} together, we know that the Hamilton cycle $H$ contains at most
\begin{align}\label{eq:tttt4}
q + d + 2\cdot\big(|W|-2d\big)+2d&\leq 2|W| - d + q\leq \frac{r-1}{r}\cdot n-2^{-10}m
\end{align}
edges that are not of colour $c^*$. The second inequality holds since $|W|=\beta n=(r-1)n/2r$, and by \eqref{eq:e(M0)}, \eqref{eq:e(M1)}, \eqref{eq:tttt1} and \eqref{eq:tttt2} we have
$$d - q\geq e(M)-2e(M')\geq2^{-9}m-2\cdot2^{-11}m \geq2^{-10}m.$$
Therefore by \eqref{eq:tttt4}, the number of $c^*$-edges in $H$ is at least
$$n-\frac{r-1}{r}\cdot n+2^{-10}m\geq \frac{n}{r}+2^{-10}m,$$
which implies $H$ has colour bias at least $2^{-10}m>b$, finishing the proof.
\end{proof}

\vspace{1.3cm}

\section*{Acknowledgements}

We are grateful to Jie Han and Rob Morris for helpful discussions, and we also thank to Rob Morris for suggestions on the presentation of this paper.

\newpage

\newcommand{\etalchar}[1]{$^{#1}$}

\section{Auxiliary Lemma Proofs}

For the reader's convenience, we will recall the statement of  \Cref{lem:properties1,lem:properties2,lem:Rjanson} and present the proofs. 

First let's focus on the quasi-random properties of $G(n,p)$ (\Cref{lem:properties1}). 

\begin{lemma}\label{lem:properties1'}
Let $c>0$ be constant, and let $R\sim G(n,p)$ with $p=c/n$.
Then
\begin{center}
$cn/4\leq e(R)\leq cn\qquad\textup{and}\qquad Y(R)\leq2p^2n^3\leq 2c^2n$ 
\end{center}
hold with probability at least $1-e^{-\Omega(n^{1/4})}$. Moreover for every $\delta>0$ with $2e/\delta<e^{c\delta/2}$, 
$$R~\ \textup{is}~\ \delta n\textup{-joined}$$
with probability at least $1-e^{-\Omega(n)}$.
\end{lemma}

We will use the Kim-Vu inequality \cite{KV} (also see~Theorem 3A in \cite{JR}) to control the upper tail of $Y(R)$. Let 
$X$ be a random variable counting the number of appearances of some fixed graph 
$H$ as a (not necessarily induced) subgraph of $R\sim G(n,p)$, Then 
$X$ can be written as a polynomial of degree $k=e(H)$ in the edge indicator variables. Let 
$$\lambda:=\E[X]\qquad \textup{and}\qquad \E_1[X]:=\max_{e\in E(K_n)} \E[X_e]$$
where $X_e$ denotes the number of copies of $H$ containing the edge $e$.
\begin{lemma}[Kim-Vu inequality]\label{lem:KimVu}
For any $\ell\geq1$, if $\lambda\geq\E_1(X)$, then
$$\Prob\Big(|X-\lambda|\geq c_k\ell^k\sqrt{\lambda\E_1[X]}\Big)\leq \exp\big(-\ell+2(k-1)\log n\big)$$
for some constant $c_k > 0$.
\end{lemma}

Using the Kim-Vu inequality, it's straightforward to prove \Cref{lem:properties1'}.

\begin{proof}[Proof of \Cref{lem:properties1'}]
First note that 
$$\E[e(R)]=\binom{n}{2}\cdot p=\frac{c(n-1)}{2}.$$ 
And by Chernoff's inequality, it follows that 
$$\Prob(cn/4\leq e(R)\leq cn)\geq 1-e^{-\Omega(n)}. $$

Next let $Z(R)$ be the number of copies of $P_2$ in $R$. Observe that deterministically
$$Y(R)\leq 2Z(R).$$
We apply the Kim-Vu inequality with $X = Z(R)$.
Note that 
$$\lambda = \E[Z(R)]\leq \binom{n}{2}\cdot n\cdot p^2\leq n^3p^2/2\leq c^2n/2,$$
and that $\E_1[Z(R)]\leq 2np\leq 2c$. Moreover, if $\ell =(4c_2)^{-1/2}\cdot(cn)^{1/4}$, where $c_2 > 0$ is the constant in \Cref{lem:KimVu}, then 
$$c_2\ell^2\sqrt{\lambda\E_1[Z(R)]}\leq c_2 \cdot (4c_2)^{-1} \cdot (cn)^{1/2}\cdot (c^2n\cdot 2c)^{1/2}\leq c^2n/2.$$
By \Cref{lem:KimVu}, it follows that
\begin{align*}
\Prob\big( Z(R) \geq c^2n \big) &\leq \Prob\big(|Z(R)-\lambda|\geq c^2n/2\big)\\
&\leq \Prob\Big(|Z(R)-\lambda| \geq c_2\ell^2\sqrt{\lambda\E_1[Z(R)]}\Big)\\
&\leq\exp\big(-\ell+2\log n\big)\\
&\leq e^{-\Omega(n^{1/4})},    
\end{align*}
and we therefore have
$$\Prob(Y(R)\geq 2c^2n)\leq \Prob(Z(R)\geq c^2n)\leq e^{-\Omega(n^{1/4})}.$$

Finally, similar to \eqref{eq:jointed}, the probability that $R$ is not $\delta n$-joined is at most 
\begin{align*}
\binom{n}{\delta n}^2(1-p)^{(\delta n)^2}\leq \bigg(\frac{e}{\delta}\bigg)^{2n}\cdot e^{-p\delta^2n^2}\leq \bigg(\frac{e}{\delta}\cdot e^{-c\delta^2/2}\bigg)^{2n}<4^{-n},
\end{align*}
since $2e/\delta<e^{c\delta/2}$, as required.
\end{proof}

Next let's consider the quasi-random properties of $G(n,p)\cup G_\alpha$ (\Cref{lem:properties2}).

\begin{lemma}
Let $0<\alpha\leq 1$ and $0<c\leq 2^{-8}$ be constants, let $G_\alpha\in \G_\alpha(n)$ and $R\sim G(n,p)$ with $p=c/n$ and $V(G_\alpha)=V(R)$. With high probability $G_\alpha\cup R$ has the following properties:

\begin{itemize}
\item[$(i)$] For every $x\in V(R)$,
$$e\big(R[N_{G_\alpha}(x)]\big)\geq \alpha^2cn/4.$$

\item[$(ii)$] For every $x,y\in V(R)$ with 
$|N_{G_\alpha}(x)\setminus N_{G_\alpha}(y)|\geq \alpha n/2$ and $|N_{G_\alpha}(y)\setminus N_{G_\alpha}(x)|\geq \alpha n/2,$ there exists a matching $M_{x,y}\subset R$ of size
$$e(M_{x,y})\geq 2^{-4}c\alpha^2 n$$
between $N_{G_\alpha}(x)\setminus N_{G_\alpha}(y)$ and $N_{G_\alpha}(y)\setminus N_{G_\alpha}(x)$.
\end{itemize}
\end{lemma}

\begin{proof}
For (i), note that for each $x\in V(G_\alpha)$ we have $$\E\big[e\big(R[N_{G_\alpha}(x)]\big)\big]\geq \binom{\alpha n}{2}\cdot p\geq \alpha^2 n/2.$$
Then we deduce $e\big(R[N_{G_\alpha}(x)]\big)\geq \alpha^2cn/4$ holds for every $x$ by Chernoff's inequality and take a union bound over $x\in V(R)$.

For (ii), fix $x,y\in V(G_\alpha)$. By assumption for $W_1:=N_{G_\alpha}(x)\setminus N_{G_\alpha}(y)$ and $W_2:=N_{G_\alpha}(y)\setminus N_{G_\alpha}(x)$, we have $|W_1|\geq \alpha n/2$ and $|W_2|\geq \alpha n/2$.
Note that the expected number of edges in $R\big[W_1,W_2\big]$ is at least
$$\E\big[e\big(R\big[W_1,W_2\big]\big)\big]\geq(\alpha n/2)^2\cdot p\geq c\alpha^2n/4.$$
Hence by apply Chernoff's inequality,
$$e\big(R\big[W_1,W_2\big]\big)\geq c\alpha^2n/8$$
holds with probability $1-e^{-\Omega(n)}$. Let $R_{x,y}\subset R$ be the induced subgraph on $W_1\cup W_2$, then the number of non-isolated edges in $R\big[W_1,W_2\big]$ is at most $Y(R_{x,y})$. Note that $R_{x,y}\sim G\big(|W_1|+|W_2|,p\big)$ and $|W_1|+|W_2|\leq 2\alpha n$, then by \Cref{lem:properties1'}, 
$$Y(R_{x,y})\leq 2\cdot p^2\cdot(|W_1|+|W_2|)^3\leq 2\cdot (c/n)^2\cdot (2\alpha n)^3\leq 16\alpha^3c^2n$$
holds with probability at least $1-e^{-\Omega(n^{1/4})}$. Now let $M_{x,y}$ be the number of isolated edges in $R[W_1,W_2]$, then
\begin{align}\label{eq:e(M_{x,y})geq calpha^2n}
e(M_{x,y})\geq c\alpha^2n/8-Y(R_{x,y})\geq c\alpha^2n/8-16\alpha^3c^2n \geq 2^{-4}c\alpha^2n    
\end{align}
since $16\alpha^3c^2n\leq2^{-4}c\alpha^2n$ by $c\leq2^{-8}$ and $\alpha\leq1$. Note \eqref{eq:e(M_{x,y})geq calpha^2n} holds with probability at least $1-e^{-\Omega(n^{1/4})}$, thus we can take a union bound over all $x,y\in V(G_\alpha)$.
\end{proof}

The following Lemma is \Cref{lem:Rjanson} in \Cref{section 4.3}.

\begin{lemma}\label{lem:appendix3}
Let $0<\beta\leq1$ be a constant and $1\ll m\leq 2^{-5}n$, and let $R\sim G(n,p)$ with $p=m/n^2$. Fix an arbitrary $U\subset V(R)$ with $|U|=\beta n$, then with high probability there exists a matching $M_U\subset R[U]$ of size
$$e(M_U)\geq 2^{-4}\beta^2m.$$
\end{lemma}
\begin{proof}
First $R[U]\sim G(\beta n, p)$, and the expected number of edges in $R[U]$ is
$$\E[e(R[U])]=\binom{\beta n}{2}\cdot p \geq \beta^2m/4.$$
Then by Chernoff's inequality with high probability 
$$e(R[U])\geq \beta^2m/8.$$
Next recall that $Y(R[U])$ is the number of edges in $R[U]$ that are not isolated.
Note that 
$$\E[Y(R[U])]\leq 2\cdot \binom{\beta n}{2}\cdot n\cdot p^2\leq \beta^2 n^3p^2\leq \beta^2m^2/n.$$

If $\beta^2m^2\ll n$, that is $m\ll n^{1/2}$, then by Markov's inequality $Y(R[U])=0$ with high probability. Hence all edges of $R[U]$ are isolated and form a matching $M\subset R[U]$ with 
$$e(M)\geq\beta^2m/8>2^{-4}\beta^2m.$$ 

Otherwise if $m=\Omega(n^{1/2})$, recall that $Z(R)$ is the number of copies of $P_2$ in $R$ and deterministically $Y(R)\leq 2Z(R)$.
We apply the Kim-Vu inequality with $X = Z(R[U])$. Note that the expectation of $Z(R[U])$
$$\E[Z(R[U])]\leq \binom{\beta n}{2}\cdot n\cdot p^2\leq \beta^2m^2/2n.$$
Let $k=2, \lambda=\beta^2m^2/2n, \ell=(\beta m)^{1/8}$ and $\E_1=\beta n\cdot p=\beta m/n$. By \Cref{lem:KimVu} we have
\begin{align*}
\Prob\big(Z(R[U])\geq\beta^2m^2/n\big)&\leq \Prob\Big(\big|Z(R[U])-\lambda\big|\geq \beta^2m^2/2n\Big)\\
&\leq\Prob\Big(\big|Z(R[U])-\lambda\big|\geq c\ell^2\cdot (\beta m)^{3/2}/n\Big)\\
&\leq \exp\big(-\ell+O(\log n)\big)\\
&\leq e^{-\Omega(n^{1/16})},    
\end{align*}
where $c$ is a constant and the second inequality holds since $$\ell^2\cdot(\beta m)^{3/2}/n\leq (\beta m)^{7/4}/n\ll (\beta m)^2/n$$
and $n$ is sufficiently large. Then
$$\Prob\big(Y(R[U])\geq 2(\beta m)^2/n\big)\leq\Prob\big(Z(R[U])\geq \beta^2 m^2/n\big)\leq e^{-\Omega(n^{1/16})}.$$
Therefore with high probability there exists a matching in $R[U]$ which contains at least
$$e(R[U])-Y(R[U])\geq \beta^2m/8- 2(\beta m)^2/n\geq 2^{-4}\beta^2m$$
edges, since $m\leq 2^{-5}n$, finishing the proof.
\end{proof}


\newcommand{\etalchar}[1]{$^{#1}$}
\begin{thebibliography}{HLM{\etalchar{+}}16}

\bibitem[BCG80]{BCG}
D.~Brada{\v{c}}, M.~Christoph, and L.~Gishboliner.
\newblock Minimum degree threshold for $ h $-factors with high discrepancy.
\newblock arXiv:2302.13780.

\bibitem[BCJP20]{BCJP}
J.~Balogh, B.~Csaba, Y.~Jing, and A.~Pluh\'{a}r.
\newblock On the discrepancies of graphs.
\newblock {\em Electron. J. Combin.}, 27(2):Paper No. 2.12, 14, 2020.

\bibitem[BDF16]{benn}
P.~Bennett, A.~Dudek, and A.~Frieze.
\newblock Adding random edges to create the square of a {H}amilton cycle.
\newblock arXiv:1710.02716.

\bibitem[BFM03]{BFM}
T.~Bohman, A.~Frieze, and R.~Martin.
\newblock How many random edges make a dense graph {H}amiltonian?
\newblock {\em Random Structures \& Algorithms}, 22(1):33--42, 2003.

\bibitem[BMPP20]{bmpp2}
J.~B{\"o}ttcher, R.~Montgomery, O.~Parczyk, and Y.~Person.
\newblock Embedding spanning bounded degree graphs in randomly perturbed graphs.
\newblock {\em Mathematika}, 66(2):422--447, 2020.

\bibitem[BTW19]{bwt2}
J.~Balogh, A.~Treglown, and A.~Wagner.
\newblock Tilings in randomly perturbed dense graphs.
\newblock {\em Combinatorics, Probability and Computing}, 28:159--176, 2019.

\bibitem[DT20]{das}
S.~Das and A.~Treglown.
\newblock Ramsey properties of randomly perturbed graphs: cliques and cycles.
\newblock {\em Combinatorics, Probability and Computing, to appear}, 2020.

\bibitem[FHLT21]{FHLT}
A.~Freschi, J.~Hyde, J.~Lada, and A.~Treglown.
\newblock A note on color-bias {H}amilton cycles in dense graphs.
\newblock {\em SIAM J. Discrete Math.}, 35(2):970--975, 2021.

\bibitem[GGS52]{GGS}
L.~Gishboliner, S.~Glock, and A.~Sgueglia.
\newblock Steiner triple systems with high discrepancy.
\newblock arXiv:2503.23252.

\bibitem[GKM22]{GKM}
L.~Gishboliner, M.~Krivelevich, and P.~Michaeli.
\newblock Color-biased {H}amilton cycles in random graphs.
\newblock {\em Random Structures \& Algorithms}, 60(3):289--307, 2022.

\bibitem[HLM{\etalchar{+}}16]{han2024colour}
H.~H{\`a}n, R.~Lang, J.~Marciano, M.~Pavez-Sign{\'e}, N.~Sanhueza-Matamala, A.~Treglown, and C.~Z{\'a}rate-Guer{\'e}n.
\newblock Colour-bias perfect matchings in hypergraphs.
\newblock arXiv:2408.11016.

\bibitem[HLMP34]{HLMP}
L.~Hollom, L.~Lichev, A.~Mond, and J.~Portier.
\newblock Discrepancies of spanning trees in dense graphs.
\newblock arXiv:2410.17034.

\bibitem[HMT21]{HPT}
J.~Han, P.~Morris, and A.~Treglown.
\newblock Tilings in randomly perturbed graphs: {B}ridging the gap between {H}ajnal-{S}zemer{\'e}di and {J}ohansson-{K}ahn-{V}u.
\newblock {\em Random Structures \& Algorithms}, 58(3):480--516, 2021.

\bibitem[JK20]{joos2}
F.~Joos and J.~Kim.
\newblock Spanning trees in randomly perturbed graphs.
\newblock {\em Random Structures \& Algorithms}, 56(1):169--219, 2020.

\bibitem[JLR11]{JLR}
S.~Janson, T.~Luczak, and A.~Rucinski.
\newblock {\em Random graphs}.
\newblock John Wiley \& Sons, 2011.

\bibitem[JR02]{JR}
S.~Janson and A.~Ruci\'{n}ski.
\newblock The infamous upper tail.
\newblock {\em Random Structures \& Algorithms}, 20(3):317--342, 2002.
\newblock Probabilistic methods in combinatorial optimization.

\bibitem[KKS17]{kks2}
M.~Krivelevich, M.~Kwan, and B.~Sudakov.
\newblock Bounded-degree spanning trees in randomly perturbed graphs.
\newblock {\em SIAM Journal on Discrete Mathematics}, 31(1):155--171, 2017.

\bibitem[KLS15]{KLS}
M.~Krivelevich, C.~Lee, and B.~Sudakov.
\newblock Long paths and cycles in random subgraphs of graphs with large minimum degree.
\newblock {\em Random Structures \& Algorithms}, 46(2):320--345, 2015.

\bibitem[KV00]{KV}
J.~Kim and V.~Vu.
\newblock Concentration of multivariate polynomials and its applications.
\newblock {\em Combinatorica}, 20(3):417--434, 2000.

\bibitem[LMX20]{LMX}
H.~Lu, J.~Ma, and S.~Xie.
\newblock Discrepancies of perfect matchings in hypergraphs.
\newblock arXiv:2408.06020.

\bibitem[Mon18]{M}
R.~Montgomery.
\newblock Topics in random graphs.
\newblock {\em Lecture notes}, 2018.

\bibitem[P\'63]{P}
L.~P\'{o}sa.
\newblock On the circuits of finite graphs.
\newblock {\em Magyar Tud. Akad. Mat. Kutat\'{o} Int. K\"{o}zl.}, 8:355--361 (1964), 1963.

\bibitem[RRS06]{RRS}
V.~R\"{o}dl, A.~Ruci\'{n}ski, and E.~Szemer\'{e}di.
\newblock A {D}irac-type theorem for 3-uniform hypergraphs.
\newblock {\em Combin. Probab. Comput.}, 15(1-2):229--251, 2006.

\end{thebibliography}
\end{document}